\UseRawInputEncoding
\documentclass[12pt,a4paper]{article}
\usepackage{latexsym,amssymb,amsfonts,amsmath,amsthm,nccmath,float,graphicx,multirow,enumitem,setspace,bm,authblk,tabularx,longtable}
\usepackage[usenames,dvipsnames]{xcolor}
\usepackage[hidelinks]{hyperref}
\usepackage[margin=2cm]{geometry}

\hypersetup{
	colorlinks,
	linkcolor={red!80!black},
	citecolor={blue!80!black},
	urlcolor={blue!80!black}
}

\newtheorem{Theorem}{Theorem}[section]
\newtheorem{Proposition}{Proposition}[section]

\newtheorem{Conjecture}{Conjecture}[section]
\newtheorem{Remark}{Remark}[section]
\newtheorem{Lemma}{Lemma}[section]

\newtheorem{Definition}{Definition}[section]
\newtheorem{Construction}{Construction}[section]

\newcommand{\Z}{\mathbb{Z}}

\newcommand{\B}{\mathcal{B}}
\newcommand{\F}{\mathcal{F}}

\def \leq {\leqslant}
\def \geq {\geqslant}

\let\oldproofname=\proofname
\renewcommand{\proofname}{\rm\bf{\oldproofname}}

\allowdisplaybreaks

\begin{document}
	
\title{Perfect difference families, perfect systems of difference sets and their applications}

\author[a]{Hengrui Liu}
\author[a]{Tao Feng}
\author[b]{Xiaomiao Wang}
\author[c]{Menglong Zhang}
\affil[a]{School of Mathematics and Statistics, Beijing Jiaotong University, Beijing, 100044, P.R. China}
\affil[b]{School of Mathematics and Statistics, Ningbo University, Ningbo, 315211, P.R. China}
\affil[c]{Institute of Mathematics and Interdisciplinary Sciences, Xidian University, Xi'an, 710126, P.R. China}

\renewcommand*{\Affilfont}{\small\it}
\renewcommand\Authands{ and }

\affil[ ]{henryleo@bjtu.edu.cn; tfeng@bjtu.edu.cn; wangxiaomiao@nbu.edu.cn; mlzhang@bjtu.edu.cn}
\date{}

\maketitle

\begin{abstract}
Let $v$ be a positive odd integer. A $(v,k,\lambda)$-perfect difference family (PDF) is a collection $\mathcal F$ of $k$-subsets of $\{0,1,\ldots,v-1\}$ such that the multiset
$\bigcup_{F\in\F}\left\{x-y : x,y\in F, x>y\right\}$
covers each element of $\left\{1,2,\ldots,(v-1)/2\right\}$ exactly $\lambda$ times. Perfect difference families are a special class of perfect systems of difference sets. They were introduced by Bermond, Kotzig, and Turgeon in the 1970s, following a problem suggested by Erd\H{o}s. In this paper, we prove that a $(v,4,\lambda)$-PDF exists if and only if $\lambda(v-1) \equiv 0 \pmod{12}$, $v \geq 13$, and $(v,\lambda) \notin \{(25,1),(37,1)\}$. This result resolves a nearly 50-year-old conjecture posed by Bermond. Perfect difference families find applications in radio astronomy, optical orthogonal codes for optical code-division multiple access systems, geometric orthogonal codes for DNA origami, difference triangle sets, additive sequences of permutations, and graceful graph labelings.

To establish our main result, we introduce a new concept termed a layered difference family. This concept provides a powerful and unified perspective that not only facilitates our proof of the main theorem but also simplifies recent existence proofs for various cyclic difference packings.
\end{abstract}


\footnotetext{Supported by NSFC under Grant 12271023, and Ningbo Natural Science Foundation under Grant 2024J018}

\section{Introduction} \label{sec:intro}

Throughout this paper, every union of (multi)-sets is understood as \emph{multiset union} with multiplicities of elements preserved. Denote by $\Z_v$ the additive group of integers modulo $v$. For positive integers $a,b$ and $c$ such that $a\leq b$ and $a\equiv b\pmod{c}$, we set $[a,b]_c:=\{a+cj:0\leq j\leq (b-a)/c\}$. When $c=1$, $[a,b]_1$ is simply written as $[a,b]$.

\begin{Definition}
Let $v$, $k$ and $\lambda$ be positive integers. A \emph{$(v,k,\lambda)$-cyclic difference packing} $($briefly CDP$)$ is a collection $\cal F$ of $k$-subsets $($called {\em base blocks}$)$ of ${\mathbb Z}_{v}$ such that
$$\Delta \F:=\bigcup_{F\in\F}\Delta F:=\bigcup_{F\in\F}\{x-y\ (\mathrm{mod}\ v) : x,y\in F, x\not=y\}$$
contains every element of ${\mathbb Z}_{v}\setminus \{0\}$ at most $\lambda$ times. If $\Delta \F$ contains every element of ${\mathbb Z}_{v}\setminus \{0\}$ exactly $\lambda$ times, then $\F$ is called a \emph{cyclic difference family} $($briefly CDF$)$.
\end{Definition}


The number of base blocks in a $(v,k,\lambda)$-CDF is $\lambda(v-1)/(k(k-1))$, and hence a necessary condition for the existence of a $(v,k,\lambda)$-CDF is $\lambda(v-1)\equiv 0\pmod{k(k-1)}$.

In the study of $(v,k,\lambda)$-CDFs, the case of $\lambda=1$ is of greatest interest, since a $(v,k,1)$-CDF can generate a cyclic Steiner system S$(2,k,v)$ without short orbits (cf. \cite[Lemma 1]{Jimbo}). It was shown by Peltesohn \cite{Peltesohn} in 1939 that a $(v,3,1)$-CDF exists if and only if $v\equiv 1,3 \pmod{6}$ and $v\neq 9$. Zhang, Feng and Wang \cite{zfw2022} recently proved that a $(v,4,1)$-CDF exists if and only if $v\equiv 1 \pmod{12}$ and $v\neq 25$. For $k\geq 5$, the existence of $(v,k,1)$-CDFs is far from being settled (cf. \cite{ab07,mr}). When $\lambda>1$, the existence of $(v,3,\lambda)$-CDFs has been completely solved \cite{cc}.

\subsection{Perfect difference families}\label{sec:1.1}

Perfect difference families are a special kind of cyclic difference families. For any base block $F$ of a $(v,k,\lambda)$-CDP $\F$, if $x,y\in F$, and $x\ (\mathrm{mod}\ v)>y\ (\mathrm{mod}\ v)$, we call $x-y\pmod{v}$ a {\em positive difference} from $F$, and $y-x \pmod{v}$ a {\em negative difference} from $F$. The collection of all positive differences (resp. negative differences) in $\Delta F$ is denoted by $\Delta^+ F$ (resp. $\Delta^- F$). Write $\Delta^+\F=\bigcup_{F\in\F} \Delta^+ F$ and $\Delta^-\F=\bigcup_{F\in\F} \Delta^- F$. Then $\Delta \F=\Delta^+ \F\cup \Delta^- \F$.

\begin{Definition}\label{def:PDF}
Let $v$ be a positive odd integer. A $(v,k,\lambda)$-CDF, $\cal F$, is called \emph{perfect}, written as a $(v,k,\lambda)$-PDF, if $\Delta^+\F$ covers every element of $\{1,2,\ldots,(v-1)/2\}$ exactly $\lambda$ times.
\end{Definition}

For example, $\{0,1,3,9\}$ is a $(13,4,1)$-CDF but not a $(13,4,1)$-PDF, while $\{0,2,5,6\}$ is a $(13,4,1)$-PDF.

Perfect difference families were introduced by Kotzig and Turgeon \cite{kt}, who referred to them as regular perfect systems of difference sets. Kotzig and Turgeon pointed out that Erd\H{o}s suggested investigating the existence problem of these families (cf. \cite{erdos,conjErdos}). They established the nonexistence of perfect difference families as follows.

\begin{Proposition}\label{thm:non-existence-PDF}{\rm \cite{kt}}
There are no $(v,k,1)$-PDFs for the following values: $\mathrm{(i)}$ $k=3$ and $v \equiv 13,19 \pmod{24}$; $\mathrm{(ii)}$ $k=4$ and $v \in \{25,37\}$; $\mathrm{(iii)}$ $k=5$ and $v \equiv 21 \pmod{40}$ or $v\in\{41,81\}$; $\mathrm{(iv)}$ $k\ge 6$.
\end{Proposition}

We remark that the nonexistence of $(81,5,1)$-PDF is not discussed in \cite{kt}. However, Laufer noted in \cite{Laufer} that Kotzig and Turgeon used an exhaustive elimination process to show that no $(81,5,1)$-PDF exists, and Laufer did a double check for this result.

Kotzig and Turgeon \cite[Theorems 2.1 and 2.2]{kt} determined the existence spectrum for $(v,3,1)$-PDFs by means of Skolem sequences. Bermond, Brouwer and Germa \cite[Theorem 1]{bbg} obtained the same result using the same technique independently.

\begin{Theorem}\label{thm:v,3,1}\emph{\cite{bbg,kt}}
A $(v,3,1)$-PDF exists if and only if $v\equiv1,7\pmod{24}$.
\end{Theorem}

Bermond \cite{b78} conjectured that a $(v,4,1)$-PDF exists for every $v\equiv 1\pmod{12}$ and $v\geq 49$ (see also \cite[Conjecture 1]{rogersaddition}). Mathon \cite{mathon} posed the existence problem of $(v,4,1)$-PDFs as an open question.

\begin{Conjecture}[Bermond, 1978]\label{conj:Berm} {\rm \cite{b78}}
There exists a $(v,4,1)$-PDF for any $v\equiv 1\pmod{12}$ and $v\geq 49$.
\end{Conjecture}

Wang and Chang \cite{wc2010} established the existence of a $(12t+1,4,1)$-PDF for any positive integer $t \leq 100$, and Ge, Miao and Sun \cite{gms} established the existence for $t \leq 1000$, both excluding $t=2,3$.

\begin{Lemma}\label{lem:k=4-small orders} {\rm \cite{gms}}
There exists a $(v,4,1)$-PDF for any positive integer $v\equiv 1\pmod{12}$ and $v\leq 12001$, with the definite exceptions of $v=25$ and $v=37$.
\end{Lemma}


So far, no infinite family of $(v,4,1)$-PDFs for which $v$ runs over a congruence class is known.

Rogers \cite{rogersmultiply} pioneered the study of $(v,k,\lambda)$-PDFs with $\lambda>1$. His definition, which requires $v$ to be odd, is the one we adopt in Definition \ref{def:PDF}. The case of even $v$ requires a modified definition (where each occurrence of $v/2$ as a positive difference must be counted twice) and falls outside the scope of our discussion.

\begin{Theorem}\label{thm:v,3,2}\emph{\cite[Theorem 4]{rogersmultiply}}
There exists a $(v,3,2)$-PDF for any positive integer $v\equiv 1\pmod{6}$.
\end{Theorem}

In Section \ref{sec:layered-DF}, we introduce a new concept called a layered difference family. This concept proves to be a powerful tool, as it provides a unified perspective on some recent existence proofs for cyclic difference packings and related topics \cite{zfw2022,zcf2025,zcf,zzcfwz2024}. These earlier proofs, which relied on direct constructive methods that appeared mysterious, now find a coherent explanation within our framework (see Remarks \ref{remark-1} and \ref{remark-2}). Using this framework, we establish in Section \ref{sec:multi} the existence of $(v,4,\lambda)$-PDFs for all positive integers $\lambda$, thereby confirming Conjecture \ref{conj:Berm}.

\begin{Theorem}\label{thm:(v,4,lambda)PDF}
For a positive odd integer $v$, a $(v,4,\lambda)$-PDF exists if and only if $\lambda(v-1)\equiv 0\pmod{12}$, $v\geq 13$, and $(v,\lambda)\not\in\{(25,1),(37,1)\}$.
\end{Theorem}


Theorem \ref{thm:(v,4,lambda)PDF} directly yields the existence of $(v,4,\lambda)$-CDFs for all $\lambda\geq 1$ in Section \ref{sec:multi-4}.

\begin{Theorem}\label{thm:4-CDF-lambda}
A $(v,4,\lambda)$-CDF exists if and only if $\lambda(v-1)\equiv 0\pmod{12}$, $v\geq 4$, and $(v,\lambda)\neq (25,1)$.
\end{Theorem}

\subsection{Perfect systems of difference sets}

Bermond, Kotzig and Turgeon \cite{bkt} introduced the notion of perfect systems of difference sets as a generalization of perfect difference families, aiming to address the problem of spacing movable antennas in radio astronomy (cf. \cite{bbr}).

In radioastronomy, a series of $m$ successive configurations is used to measure the spatial frequencies from a specific area of the sky. In the $i$-th configuration, $k_i$ movable antennas are employed. The distances between these antennas determine the set of spatial frequencies that can be sampled. Collectively, the $m$ configurations can yield up to $\sum_{i=1}^{m} \binom{k_i}{2}$ unique baselines. To maximize efficiency, the distances between antennas must be chosen such that they form a contiguous set of multiples of a given unit length, ensuring there are no gaps and that no measurement is repeated.

\begin{Definition}
Let $m,k_{1}, k_{2}, \dots, k_{m}$ and $c$ be positive integers. A \emph{perfect system of difference sets} with \emph{threshold} $c$, denoted by $(m,\{k_{1},\dots,k_{m}\},c)$-PSDS, is a collection $\mathcal{A}$ of $m$ subsets $($called \emph{base blocks}$)$ of nonnegative integers, where the $i$-th subset has size $k_i$, $1\leq i\leq m$, such that
$$\Delta^+ \mathcal A:=\bigcup_{A\in\mathcal A}\Delta^+ A:=\bigcup_{A\in\mathcal A}\{x-y : x,y\in A, x>y\}$$
covers every element of $\{c, c+1, \dots, c-1+\sum_{i=1}^{m} \binom{k_i}{2} \}$ exactly once. If $k_{1} = k_{2} = \cdots = k_{m} = k$, then such a system is \emph{regular}, and is abbreviated as an $(m, k, c)$-PSDS.
\end{Definition}

When $c=1$, an $(m,k,1)$-PSDS yields an $(mk(k-1)+1,k,1)$-PDF, and conversely, a $(v,k,1)$-PDF leads to a $((v-1)/(k(k-1)), k, 1)$-PSDS.

The following proposition gives necessary conditions for the existence of PSDSs.

\begin{Proposition}\label{thm:PSDS-nece}
\begin{enumerate}
    \item[$(1)$] {\rm \cite[Proposition 1]{bbg}} If there exists an $(m,3,c)$-PSDS, then $m\geq 2c-1$.  Furthermore, $m \equiv 0,1 \pmod{4}$ if $c$ is odd, and $m \equiv 0,3 \pmod{4}$ if $c$ is even.
    \item[$(2)$] {\rm \cite[Proposition 2.1, 2.2]{bkt}} If there exists an $(m,4,c)$-PSDS, then $m\geq 2c-1$. No $(m,k,c)$-PSDS exists for any $k\geq 6$.
    \item[$(3)$] {\rm \cite[Theorem 2]{kl}} If there exists an $(m,5,c)$-PSDS, then $m \geq 4c$.
\end{enumerate}
\end{Proposition}

Following the work of Bermond, Brouwer and Germa \cite{bbg}, Rogers \cite{rogers86} established the existence of an $(m,3,c)$-PSDS for any $c\geq 1$ by employing Langford sequences developed in \cite{Simpson}.

\begin{Theorem}\label{thm:3-PSDS} {\rm \cite[Theorem 5]{rogers86}}
For each $c\geq 1$, an $(m,3,c)$-PSDS exists if and only if $m\geq 2c-1$, and $m \equiv 0,1 \pmod{4}$ if $c$ is odd, and $m \equiv 0,3 \pmod{4}$ if $c$ is even.
\end{Theorem}

This paper examines the existence of $(m,4,3)$-PSDSs. On one hand, Bermond, Kotzig and Turgeon \cite{bkt} pointed out that from a practical point of view, $c=2$ or $3$ is desirable. On the other hand, they observed that every $(m,4,3)$-PSDS implies an $(m+1,4,1)$-PSDS, i.e., a $(12m+13,4,1)$-PDF.

\begin{Lemma}\label{lem:c=3 to 1}{\rm \cite[Proposition 4.2]{bkt}}
If there exists an $(m,4,3)$-PSDS, then there exists an $(m+1,4,1)$-PSDS.
\end{Lemma}

\begin{proof}
Let $\cal A$ be an $(m,4,3)$-PSDS. Then $\mathcal A\cup \{\{0,1,6m+4,6m+6\}\}$ forms an $(m+1,4,1)$-PSDS.
\end{proof}

We prove the following theorem in Section \ref{sec:main-construction}.

\begin{Theorem}\label{thm:(m,4,3)-PSDS}
An $(m,4,3)$-PSDS exists if and only if $m\geq 5$.
\end{Theorem}

\subsection{Organization of this paper}

This paper is structured as follows. Section \ref{sec:3-PDF} presents a new construction for $(v,3,1)$-PDFs. We then abstract this construction in Section \ref{sec:layered-DF} to define layered difference families, a powerful new tool for constructing cyclic difference packings. This tool is then employed in Section \ref{sec:main-construction} to provide a direct construction of $(m,4,3)$-PSDSs for all $m\geq 5$, thus proving Theorem \ref{thm:(m,4,3)-PSDS}. The same method is applied in Section \ref{sec:multi} to establish the existence of $(v,4,\lambda)$-PDFs for all $\lambda\geq 1$. Section \ref{sec:application} explores applications of perfect difference families to additive sequences of permutations, perfect difference matrices, difference triangle sets, optical orthogonal codes, geometric orthogonal codes, and graceful graph labelings. Concluding remarks are provided in Section \ref{con:conluding}.

\section{A direct construction for $(v,3,1)$-PDFs}\label{sec:3-PDF}

We shall introduce a new class of difference families, termed layered difference families, in Section \ref{sec:layered-DF}. These families serve as a powerful tool for constructing cyclic difference packings. Given that their definition is significantly more intricate than that of classical difference families, we first present a new construction for $(v,3,1)$-PDFs in this section, and then abstract this construction to formally define layered difference families in the next section.

As noted in Section \ref{sec:1.1}, the existence of $(v,3,1)$-PDFs was established independently by Kotzig and Turgeon \cite{kt} and by Bermond, Brouwer, and Germa \cite{bbg} using Skolem sequences. In fact, an earlier direct construction for cyclic Steiner triple systems was given by Hwang and Lin \cite{hl}. One can verify that their base blocks actually form a $(v,3,1)$-PDF for all $v \equiv 1, 7 \pmod{24}$. In this section, we present a new direct construction for $(v,3,1)$-PDFs. Our approach differs from that of Hwang and Lin and offers a more concise alternative.

\begin{table}[!b]\centering
\renewcommand{\arraystretch}{1.1}
\begin{tabular}{|c|c|c|}\hline
$\Delta^+ F_{1,i}$  & $\Delta^+{\cal F}_1$ & $\Delta^-{\cal F}_1$ \\\hline					
$3t+i$              & $[3t+1,4t-1]\setminus \{\frac{7t}{2}\}$ & $[8t+2,9t]\setminus \{\frac{17t}{2}+1\}$ \\
$2i$                & $[2,2t-2]_{2}\setminus \{t\}$ & $[10t+3,12t-1]_{2}\setminus \{11t+1\}$ \\
$3t-i$              & $[2t+1,3t-1]\setminus \{\frac{5t}{2}\}$ & $[9t+2,10t]\setminus\{\frac{19t}{2}+1\}$ \\\hline \hline
$\Delta^+ F_{2,i}$ & $\Delta^+{\cal F}_2$ & $\Delta^-{\cal F}_2$  \\\hline
$5t+i$ & $[5t,6t-1]$ & $[6t+2,7t+1]$ \\
$2i+1$ & $[1,2t-1]_{2}$ & $[10t+2,12t]_{2}$ \\
$5t-i-1$ & $[4t,5t-1]$ & $[7t+2,8t+1]$ \\\hline
\end{tabular} \hspace{0.6cm}
\begin{tabular}{|c|c|}\hline
$\Delta^+ B_{1}$ & $\Delta^+ B_{2}$ \\\hline
$t$ & $\frac{5t}{2}$    \\
$3t$ & $6t$   \\
$2t$ & $\frac{7t}{2}$  \\\hline
\end{tabular}
\caption{Differences from base blocks in Proposition \ref{prop:3-PDF}} \label{tab:3-PDF-1}
\end{table}

\begin{Proposition}\label{prop:3-PDF}
For any positive integer $v\equiv 1 \pmod{24}$, there exists a $(v,3,1)$-PDF.
\end{Proposition}

\begin{proof}
Let $v=12t+1$, where $t\geq 2$ is even. A $(12t+1,3,1)$-PDF consists of $2t$ base blocks. The first $2t-2$ base blocks are listed below:
\begin{center}
\begin{tabular}{llll}
$F_{1,i}=\{0$, & $3t+i$, & $2i\}$, & $1\leq i\leq t-1$ and $i\neq t/2$; \\
$F_{2,i}=\{0$, & $5t+i$, & $2i+1\}$, & $0\leq i\leq t-1$. \\
\end{tabular}
\end{center}
The remaining 2 base blocks are $B_{1}=\{0,t,3t\}$ and $B_{2}=\{0,\frac{5t}{2},6t\}$.

To facilitate verification, all differences arising from these base blocks are presented in Table \ref{tab:3-PDF-1}, where $\F_1=\{F_{1,i}: 1\leq i\leq{t-1}, i\neq t/2\}$ and $\F_2=\{F_{2,i}: 0\leq i\leq{t-1}\}$. The positive differences cover each element in the set $\{1,2,\ldots,(v-1)/2\}$ exactly once.
\end{proof}

In the proof of Proposition \ref{prop:3-PDF}, except for the two sporadic base blocks $B_1$ and $B_2$, all the remaining base blocks of a $(v,3,1)$-PDF are divided into two parts, denoted as $\F_1$ and $\F_2$. Each base block in $\F_r$ with $r=1,2$ is of the form
\begin{align}\label{eqn:1}
\{b_{r0}\cdot t+0\cdot i+\epsilon_{r0},\ \ \ \ \ b_{r1}\cdot t+1\cdot i+\epsilon_{r1},\ \ \ \ \ b_{r2}\cdot t+2\cdot i+\epsilon_{r2}\}, \tag{$\dagger$}
\end{align}
where $i$ runs over some set $I_r$. Without loss of generality, we set $b_{r0}=\epsilon_{r0}=0$.

Our construction consists of two steps.  The first step is to select appropriate parameters $b_{r1}$, $b_{r2}$, $\epsilon_{r1}$, $\epsilon_{r2}$, and $I_r$ to generate almost all base blocks of a $(v,3,1)$-PDF. This is done in two stages: first, choosing $b_{r1}$ and $b_{r2}$ (as coefficients of $t$) to form a layered difference family, defined in Section \ref{sec:layered-DF}, and then selecting $\epsilon_{r1}$, $\epsilon_{r2}$, and $I_r$. Typically, $I_1$ and $I_2$ are selected as subsets of the integer interval $[1,t]$ by removing a few values from its two endpoints or several values from its middle.

Once $b_{r1}$ and $b_{r2}$ are determined, we can choose $\epsilon_{r1}$ and $\epsilon_{r2}$. The parameters $\epsilon_{r1}$ and $\epsilon_{r2}$ serve to shift the intervals formed by the differences, ensuring that these intervals are pairwise disjoint. Consequently, they are typically small integers that can be selected manually or with the aid of a computer. We illustrate the selection of $\epsilon_{r1}$ and $\epsilon_{r2}$ by constructing a $(v,3,1)$-PDF with $v\equiv 7 \pmod{24}$ in the following proposition. The method for selecting $b_{r1}$ and $b_{r2}$ will be presented after this proposition.

\begin{Proposition}\label{porp:3-PDF-1}
For any positive integer $v\equiv 7 \pmod{24}$ and $v\geq 55$, there exists a $(v,3,1)$-PDF.
\end{Proposition}

\begin{proof}
Let $v=12t+7$, where $t\geq 4$ is even. A $(12t+7,3,1)$-PDF consists of $2t+1$ base blocks. The first $2t-6$ base blocks are listed below:
\begin{center}
\begin{tabular}{llll}
$F_{1,i}=\{0$, &$3t+i+1$, &$2i\}$, & $2\leq i\leq t-2$ and $i\neq t/2$; \\
$F_{2,i}=\{0$, &$5t+i+3$, &$2i+1\}$, & $1\leq i\leq t-2$.
\end{tabular}
\end{center}
The remaining 7 base blocks are
\begin{center}
\begin{tabular}{lllll}
$\{0,1,2t-1\}$,
&$\{0,2,4t+2\}$,
&$\{0,2t,4t+1\}$,
&$\{0,2t+2,5t+2\}$,
&$\{0,\frac{5t}{2}+1,6t+2\}$,\\
$\{0,t,5t+3\}$,
&$\{0,3t+1,6t+3\}$.
\end{tabular}
\end{center}

Here, both $F_{1,i}$ and $F_{2,i}$ are of the form specified in \eqref{eqn:1}, with the parameters $b_{r1}$ and $b_{r2}$ set to the same values as those in the proof of Proposition \ref{prop:3-PDF}. To illustrate how to choose $\epsilon_{r1}$ and $\epsilon_{r2}$ after fixing $b_{r1}$ and $b_{r2}$, we rewrite $F_{1,i}$ and $F_{2,i}$ as follows:
\begin{center}
\begin{tabular}{llll}
$F_{1,i}=\{0$, &$3t+i+\epsilon_{11}$, &$2i+\epsilon_{12}\}$, & $2\leq i\leq t-2$ and $i\neq t/2$; \\
$F_{2,i}=\{0$, &$5t+i+\epsilon_{21}$, &$2i+\epsilon_{22}\}$, & $1\leq i\leq t-2$.
\end{tabular}
\end{center}
We list all positive differences arising from $\F_1$ and $\F_2$ in Table \ref{tab:3-PDF-2-2}, where $\F_1=\{F_{1,i}: 2\leq i\leq{t-2}, i\neq t/2\}$ and $\F_2=\{F_{2,i}: 1\leq i\leq{t-2}\}$.

\begin{table}[htb]\centering
\renewcommand{\arraystretch}{1.2}
\begin{tabular}{|c|c|}\hline
$\Delta^+ F_{1,i}$  & $\Delta^+{\cal F}_1$  \\\hline					
$3t+i+\epsilon_{11}$              & $[3t+2+\epsilon_{11},4t-2+\epsilon_{11}]\setminus \{\frac{7t}{2}+\epsilon_{11}\}$  \\
$2i+\epsilon_{12}$                & $[4+\epsilon_{12},2t-4+\epsilon_{12}]_{2}\setminus \{t+\epsilon_{12}\}$  \\
$3t-i+\epsilon_{11}-\epsilon_{12}$              & $[2t+2+\epsilon_{11}-\epsilon_{12},3t-2+\epsilon_{11}-\epsilon_{12}]\setminus \{\frac{5t}{2}+\epsilon_{11}-\epsilon_{12}\}$ \\\hline \hline
$\Delta^+ F_{2,i}$ & $\Delta^+{\cal F}_2$  \\\hline
$5t+i+\epsilon_{21}$ & $[5t+1+\epsilon_{21},6t-2+\epsilon_{21}]$  \\
$2i+\epsilon_{22}$ & $[2+\epsilon_{22},2t-4+\epsilon_{22}]_{2}$  \\
$5t-i+\epsilon_{21}-\epsilon_{22}$ & $[4t+2+\epsilon_{21}-\epsilon_{22},5t-1+\epsilon_{21}-\epsilon_{22}]$ \\\hline
\end{tabular}
\caption{Positive differences from the first $2t-6$ base blocks in Proposition \ref{porp:3-PDF-1}} \label{tab:3-PDF-2-2}
\end{table}

First, we need to ensure that the intervals $[4+\epsilon_{12},2t-4+\epsilon_{12}]_{2}$ and $[2+\epsilon_{22},2t-4+\epsilon_{22}]_{2}$ are disjoint. To achieve this, $\epsilon_{12}$ and $\epsilon_{22}$ must have different parity. Here we take $\epsilon_{12}=0$ and $\epsilon_{22}=1$. Furthermore, by setting $\epsilon_{11}=1$ and $\epsilon_{21}=3$, and then conducting a computer search, we can obtain the remaining 7 base blocks. Note that the choices of $\epsilon_{11}$ and $\epsilon_{21}$ are not unique and may not be successful on the first attempt. The strategy is to first try small values for $\epsilon_{11}$ and $\epsilon_{21}$. If the remaining base blocks can be found, we finalize the choice. Otherwise, we readjust the values $\epsilon_{11}$ and $\epsilon_{21}$, and if necessary, even modify the ranges of $I_1$ and $I_2$.

The verification of the correctness of these base blocks is left to the reader.
\end{proof}

We now describe a method for selecting $b_{r1}$ and $b_{r2}$. We analyze the coefficients of $t$ in \eqref{eqn:1}. These coefficients form two ordered 3-multi-subsets of $\Z_{12}$ in the proof of Propositions \ref{prop:3-PDF} and \ref{porp:3-PDF-1}:
\begin{center}
\begin{tabular}{lll}
$(b_{10},b_{11},b_{12})=(0,3,0)$; \\
$(b_{20},b_{21},b_{22})=(0,5,0)$.
\end{tabular}
\end{center}
Let us analyze the ``differences'' produced by the ordered pairs $(b_{r0},b_{r1})$, $(b_{r1},b_{r0})$, $(b_{r0},b_{r2})$, $(b_{r2},b_{r0})$, $(b_{r1},b_{r2})$ and $(b_{r2},b_{r1})$ for $r=1,2$. Informally, the  ``difference'' for an ordered pair is defined as the right-hand value minus the left-hand value, followed by an additional shift.

Let us examine the $(v,3,1)$-PDFs given in the proof of Proposition \ref{prop:3-PDF}. Observe the row labeled $3t+i$ in Table \ref{tab:3-PDF-1}, which comes from the value $(b_{11}t+i)-b_{10}t$. For the ordered pair $(b_{10},b_{11})=(0,3)$, the difference $b_{11}-b_{10}\equiv 3\pmod{12}$ corresponds to the interval $[3t+1,4t-1]$. Here, {\bf we denote an interval by the coefficient of $t$ in its starting point}. For the reversed pair $(b_{11},b_{10})=(3,0)$, we compute the difference $b_{10}-b_{11}\equiv 9\pmod{12}$. A shift of $-1$ is then applied modulo $12$ to align with our convention of denoting intervals by the coefficient of $t$ in their starting point, yielding $9-1\equiv 8\pmod{12}$. This value $8$ corresponds to the interval $[8t+2, 9t]$, which is the additive inverse modulo $12t+1$ of the interval $[3t+1, 4t-1]$.

Observe the row labeled $5t+i$ in Table \ref{tab:3-PDF-1}, which comes from the value $(b_{21}t+i)-b_{20}t$. For the ordered pair $(b_{20},b_{21})=(0,5)$, the difference $b_{21}-b_{20}\equiv 5\pmod{12}$ corresponds to the interval $[5t,6t-1]$. For the reversed pair $(b_{21},b_{20})=(5,0)$, we compute the difference $b_{20}-b_{21}\equiv 7\pmod{12}$, and then apply a shift of $-1$ to obtain $7-1\equiv 6\pmod{12}$. This value $6$ corresponds to the interval $[6t+2, 7t+1]$, which is the additive inverse modulo $12t+1$ of the interval $[5t, 6t-1]$.

Observe the row labeled $3t-i$ in Table \ref{tab:3-PDF-1}, which comes from the value $(b_{11}t+i)-(b_{12}t+2i)$. For the ordered pair $(b_{11},b_{12})=(3,0)$, the difference $b_{12}-b_{11}\equiv 9\pmod{12}$ corresponds to the interval $[9t+2,10t]$. For the reversed pair $(b_{12},b_{11})=(0,3)$, we compute the difference $b_{11}-b_{12}\equiv 3\pmod{12}$, and then apply a shift of $-1$ to obtain $3-1\equiv 2\pmod{12}$. This value $2$ corresponds to the interval $[2t+1,3t-1]$, which is the additive inverse modulo $12t+1$ of the interval $[9t+2,10t]$.

Observe the row labeled $5t-i-1$ in Table \ref{tab:3-PDF-1}, which comes from the value $(b_{21}t+i)-(b_{22}t+2i+1)$. For the ordered pair $(b_{21},b_{22})=(5,0)$, the difference $b_{22}-b_{21}\equiv 7\pmod{12}$ corresponds to the interval $[7t+2,8t+1]$. For the reversed pair $(b_{22},b_{21})=(0,5)$, we compute the difference $b_{21}-b_{22}\equiv 5\pmod{12}$, and then apply a shift of $-1$ to obtain $5-1\equiv 4\pmod{12}$. This value $4$ corresponds to the interval $[4t,5t-1]$, which is the additive inverse modulo $12t+1$ of the interval $[7t+2,8t+1]$.

\begin{table}[!b]\centering
\renewcommand{\arraystretch}{1.2}
\begin{tabular}{|c|c|c|}\hline
$\Delta^+{\cal F}_1$ & Coefficient of $t$ in the starting point & Source
\\\hline					
$[3t+1,4t-1]\setminus\{\frac{7t}{2}\} $ & $3$ & $b_{11}-b_{10}$ \\
$[2, 2t-2]_{2}$ & $0$* & $b_{12}-b_{10}$ \\
                & $1$* & $b_{12}-b_{10}+1$ \\
$[2t+1,3t-1]\setminus\{\frac{5t}{2}\}$ & $2$ & $b_{11}-b_{12}-1$\\
				\hline
$\Delta^-{\cal F}_1$ & & \\
				\hline
$ [8t+2,9t]\setminus\{\frac{17t}{2}+1\} $ & $8$ & $b_{10}-b_{11}-1$ \\
$[10t+3, 12t-1]_{2}$ & $11$* & $b_{10}-b_{12}-1$ \\
                     & $10$* & $b_{10}-b_{12}-2$ \\
$[9t+2,10t]\setminus\{\frac{19t}{2}+1\}$ & $9$ & $b_{12}-b_{11}$ \\
				\hline \hline
$\Delta^+{\cal F}_2$ & Coefficient of $t$ in the starting point & \\
				\hline
$[5t,6t-1]$ & $5$ & $b_{21}-b_{20}$ \\
$[1,2t-1]_{2}$ & $0$* & $b_{22}-b_{20}$ \\
               & $1$* & $b_{22}-b_{20}+1$ \\
$[4t,5t-1]$ & $4$ & $b_{21}-b_{22}-1$ \\\hline
$\Delta^-{\cal F}_2$ &  & \\
				\hline
$[6t+2,7t+1]$ & $6$ & $b_{20}-b_{21}-1$ \\
$[10t+2,12t]_{2}$ &$11$* & $b_{20}-b_{22}-1$ \\
&$10$* & $b_{20}-b_{22}-2$ \\
$[7t+2,8t+1]$ & $7$ & $b_{22}-b_{21}$ \\\hline
\end{tabular}
\caption{Coefficients of $t$ in the starting points of the intervals in Table \ref{tab:3-PDF-1}} \label{tab:auxtab}
\end{table}

Observe the row labeled $2i$ in Table \ref{tab:3-PDF-1}, which comes from the value $(b_{12}t+2i)-b_{10}t$. For the pair $(b_{10},b_{12})=(0,0)$, the difference $b_{12}-b_{10}\equiv 0\pmod{12}$ corresponds to the interval $[2,2t-2]_2=[2,t]_2\cup[t+2,2t-2]_2$. According to our convention for interval labeling, the two intervals $[2,t]_2$ and $[t+2,2t-2]_2$ can be represented by the differences $b_{12}-b_{10}\pmod{12}$ and $b_{12}-b_{10}+1\pmod{12}$, respectively. Note that each of the two intervals is only half-covered. For the pair $(b_{12},b_{10})=(0,0)$, we compute the difference $b_{10}-b_{12}\equiv 0\pmod{12}$, and then apply a shift of $-1$ to obtain $0-1\equiv 11\pmod{12}$ and a shift of $-2$ to obtain $0-2\equiv 10\pmod{12}$. The values $10$ and $11$ correspond to the interval $[10t+3,12t-1]_2$, which is the additive inverse modulo $12t+1$ of the interval $[2,2t-2]_2$.

Observe the row labeled $2i+1$ in Table \ref{tab:3-PDF-1}, which comes from the value $(b_{22}t+2i+1)-b_{20}t$. For the pair $(b_{20},b_{22})=(0,0)$, the difference $b_{22}-b_{20}\equiv 0\pmod{12}$ corresponds to the interval $[1,2t-1]_2=[1,t-1]_2\cup[t+1,2t-1]_2$. The two intervals $[1,t-1]_2$ and $[t+1,2t-1]_2$ can be represented by the differences $b_{22}-b_{20}\pmod{12}$ and $b_{22}-b_{20}+1\pmod{12}$, respectively. For the pair $(b_{22},b_{20})=(0,0)$, we compute the difference $b_{20}-b_{22}\equiv 0\pmod{12}$, and then apply a shift of $-1$ to obtain $0-1\equiv 11\pmod{12}$ and a shift of $-2$ to obtain $0-2\equiv 10\pmod{12}$. The values $10$ and $11$ correspond to the interval $[10t+2,12t]_2$, which is the additive inverse modulo $12t+1$ of the interval $[1,2t-1]_2$.

Now we summarize all the intervals of differences arising from the ordered pairs $(b_{r0},b_{r1})$, $(b_{r1},b_{r0})$, $(b_{r0},b_{r2})$, $(b_{r2},b_{r0})$, $(b_{r1},b_{r2})$ and $(b_{r2},b_{r1})$ for $r=1,2$ in Table \ref{tab:auxtab}, where the symbol $*$ denotes the intervals that are half-covered. We can see that the two ordered 3-multi-subsets of $\mathbb{Z}_{12}$,  $(b_{10},b_{11},b_{12})$ and $(b_{20},b_{21},b_{22})$, must satisfy that the ``differences'' generated by them, which are listed in the third column of the table, cover each of the $12$ integers from $0$ to $11$ exactly once (note that the differences marked with $*$ are counted as half). This motivates us to introduce the concept of layered difference families in the next section.

\section{Layered difference families}\label{sec:layered-DF}

Our strategy for constructing a $(v,k,\lambda)$-CDF is to first construct most of its base blocks. Each base block is of the form:
\begin{align}\label{eqn:general form}
\{b_{r,0}\cdot t+0\cdot i+\epsilon_{r,0},\ \ \ \ b_{r,1}\cdot t+1\cdot i+\epsilon_{r,1},\ \ \ \ldots, \ \ \ b_{r,k-1}\cdot t+(k-1)\cdot i+\epsilon_{r,k-1}\}, \tag{$\ddagger$}
\end{align}
where $i$ runs over some set $I_r$. Then we construct the remaining sporadic base blocks. To select appropriate parameters $b_{r,j}$, $0\leq j\leq k-1$, we introduce the concept of layered difference families.

We shall use the standard notation for group rings as given in \cite[Chapter VI.3]{bjh}. Let $\mathbb{Q}$ be the field of rational numbers and $\mathbb{Q}[\mathbb{Z}_v]$ be the group ring of the cyclic group $(\mathbb{Z}_v, +)$ over $\mathbb{Q}$. Specifically, the group ring $\mathbb{Q}[\mathbb{Z}_v]$ consists of all formal sums of the form $\sum_{g \in \mathbb{Z}_v} a_g g$, where $a_g \in \mathbb{Q}$. To distinguish between the addition operation in $\mathbb{Q}$ and that in $\mathbb{Z}_v$, we denote the addition operation in $\mathbb{Q}$ as $\oplus$. This induces the addition operation on the group ring $\mathbb{Q}[\mathbb{Z}_v]$ (i.e., the operation represented by the summation symbol $\sum$ in the formal sums) as follows
$$\left(\sum_{g\in \mathbb{Z}_v} a_g  g \right) \oplus \left(\sum_{g \in \mathbb{Z}_v} a'_g  g \right) = \sum_{g \in \mathbb{Z}_v} (a_g \oplus a'_g) g.$$


\begin{Definition}\label{def-LDF}
Let $k\ge 2$ be an integer, and let $B=(b_{0},b_{1},\ldots,b_{k-1})$ be an ordered $k$-multi-subset of $(\Z_v,+)$. Define
$$\Delta^* B:=\sum_{0\le s_1 < s_2 < k} \sum_{0\le \ell < s_2-s_1} \underline{\frac{1}{s_2-s_1}}\left[(b_{s_2} - b_{s_1} +\ell) \oplus (b_{s_1} - b_{s_2} -\ell - 1)\right]\in \mathbb{Q}[\Z_v],$$
where the underlined $\underline{\frac{1}{s_2-s_1}}$ emphasizes that the value $\frac{1}{s_2-s_1}$ is taken from $\mathbb{Q}$. Let $\B$ be a collection of ordered $k$-multi-subsets $($called \emph{base blocks}$)$ of $\Z_v$. If
$$\Delta^* \B := \sum_{B\in \B} \Delta^* B=\underline{\lambda} \sum_{g\in \Z_{v}} g,$$
then $\B$ is called a {\em layered difference family}, written as a $(v,k,\lambda)$-LDF. The values $b_{s_2} - b_{s_1} +\ell$ and $b_{s_1} - b_{s_2} -\ell - 1$ are called \emph{differences}, and $\frac{1}{s_2-s_1}$ is their \emph{coefficient}.
\end{Definition}

In Definition \ref{def-LDF}, the term $\underline{\frac{1}{s_2-s_1}}[(b_{s_2}-b_{s_1}+\ell) \oplus (b_{s_1} - b_{s_2} -\ell - 1)]$ arises from the difference intervals generated by the ordered pairs $(b_{s_1}, b_{s_2})$ and $(b_{s_2}, b_{s_1})$, which correspond to the ordered pairs  $(b_{s_1}t+s_1i+\epsilon_{s_1}, b_{s_2}t+s_2i+\epsilon_{s_2})$ and $(b_{s_2}t+s_2i+\epsilon_{s_2}, b_{s_1}t+s_1i+\epsilon_{s_1})$ in \eqref{eqn:general form}.

\begin{Lemma}\label{lem:12,3,1-LDF}
There exists a $(12,3,1)$-LDF.
\end{Lemma}

\begin{proof}
Let $B_1=(0,3,0)$ and $B_2=(0,5,0)$ be two ordered $3$-multi-subsets of $\Z_{12}$, which have been used to construct $(v,3,1)$-PDFs in Section \ref{sec:3-PDF}. Write $B_{1}=(b_{10},b_{11},b_{12})=(0,3,0)$. Then
\begin{align*}
\Delta^* B_{1}=&\ \underline{1}[(b_{11}-b_{10}) \oplus (b_{10}-b_{11}-1)]\oplus\\
&\  \underline{1}[(b_{12}-b_{11}) \oplus (b_{11}-b_{12}-1)] \oplus \\
&\ \underline{\frac{1}{2}}[(b_{12}-b_{10}) \oplus (b_{10}-b_{12}-1) \oplus (b_{12}-b_{10}+1) \oplus (b_{10}-b_{12}-1-1)]\\
=&\ \underline{1}[3\oplus(-4)] \oplus \underline{1}[(-3)\oplus2] \oplus \underline{\frac{1}{2}}[0\oplus(-1)\oplus1\oplus(-2)]\\
=&\ \underline{1}[3\oplus8] \oplus \underline{1}[9\oplus2] \oplus \underline{\frac{1}{2}}[0\oplus11\oplus1\oplus10].
\end{align*}
Similarly,
\begin{align*}
\Delta^* B_{2}=\underline{1}[5 \oplus 6] \oplus \underline{1}[7 \oplus 4] \oplus \underline{\frac{1}{2}}[0 \oplus 11 \oplus 1 \oplus 10].
\end{align*}
Since $\Delta^* B_{1} \oplus \Delta^* B_{2} = \sum_{g\in \Z_{12}} g$, it follows that $B_{1} $ and $ B_{2}$ form a $(12,3,1)$-LDF.
\end{proof}

\begin{Lemma}\label{lem:72-LDF}
There exists a $(72,4,1)$-LDF.
\end{Lemma}

\begin{proof}
Take the following ordered 4-multi-subsets of $\Z_{72}$:
\begin{center}
\begin{tabular}{llll}
$B_{1}=(0,43,31,8)$, &\ \ \ \ \ \  $B_{2}=(0,23,5,8)$, &\ \ \ \ \ \  $B_{3}=(0,41,25,8)$, \\
$B_{4}=(0,35,5,0)$, &\ \ \ \ \ \  $B_{5}=(0,47,19,0)$, &\ \ \ \ \ \  $B_{6}=(0,21,13,0)$.
\end{tabular}
\end{center}
Write $B_1=(b_{10},b_{11},b_{12},b_{13})$. Then

\begin{align*}
\Delta^* B_{1} = &\ \underline{1}\ [(b_{11}-b_{10}) \oplus (b_{10}-b_{11}-1)]\oplus \\
& \ \underline{1}\ [(b_{12}-b_{11}) \oplus (b_{11}-b_{12}-1)] \oplus \\
& \ \underline{1}\ [(b_{13}-b_{12}) \oplus (b_{12}-b_{13}-1)] \oplus \\
& \ \underline{\frac{1}{2}}\ [(b_{12}-b_{10}) \oplus (b_{10}-b_{12}-1) \oplus (b_{12}-b_{10}+1) \oplus (b_{10}-b_{12}-1-1)] \oplus \\
&\ \underline{\frac{1}{2}}\ [(b_{13}-b_{11}) \oplus (b_{11}-b_{13}-1) \oplus (b_{13}-b_{11}+1) \oplus (b_{11}-b_{13}-1-1)] \oplus \\
&\ \underline{\frac{1}{3}}\ [(b_{13}-b_{10}) \oplus (b_{10}-b_{13}-1) \oplus (b_{13}-b_{10}+1) \oplus \\ & \ \ \ \ \ \ \ \ \ \ \ \ \ \ \ \ \ \ \ \ (b_{10}-b_{13}-1-1) \oplus (b_{13}-b_{10}+2) \oplus (b_{10}-b_{13}-2-1)] \\
=&\ \underline{1}\ [43 \oplus (-44)] \oplus \underline{1}\ [(-12) \oplus 11] \oplus \underline{1}\ [(-23) \oplus 22] \oplus \\
&\ \underline{\frac{1}{2}}\ [31 \oplus (-32) \oplus 32 \oplus (-33)] \oplus \underline{\frac{1}{2}}\ [ (-35) \oplus 34 \oplus (-34) \oplus 33] \oplus \\
&\ \underline{\frac{1}{3}}\ [8 \oplus (-9) \oplus 9 \oplus (-10) \oplus 10 \oplus (-11)] \\
=&\ \underline{1}\ [43 \oplus 28] \oplus \underline{1}\ [60 \oplus 11] \oplus \underline{1}\ [49 \oplus 22] \oplus \\
&\ \underline{\frac{1}{2}}\ [31 \oplus 40 \oplus 32 \oplus 39] \oplus \underline{\frac{1}{2}}\ [37 \oplus 34 \oplus 38 \oplus 33] \oplus \\
&\ \underline{\frac{1}{3}}\ [8 \oplus 63 \oplus 9 \oplus 62 \oplus 10 \oplus 61].
\end{align*}
Similarly, one can compute $\Delta^* B_{j}$ for $2\leq j\leq 6$ and verify that $\sum_{j=1}^{6} \Delta^* B_{j}=\sum\limits_{g\in \Z_{72}} g$. Thus $B_{1},B_{2},\dots,B_{6}$ form a $(72,4,1)$-LDF.
\end{proof}

\begin{Remark}\label{remark-1}
The $(72,4,1)$-LDF given in the proof of Lemma $\ref{lem:72-LDF}$ was implicitly used in \textnormal{\cite{zfw2022}} to establish the existence of $(v,4,1)$-CDFs. In that paper, the constructed $(v,4,1)$-CDFs have base blocks that almost all follow the form of \eqref{eqn:general form}, where the coefficients $b_{r,j}$ of $t$ constitute a $(72,4,1)$-LDF. Previously, the theoretical rationale for selecting these coefficients was unclear, as they were obtained solely through computer search. We now recognize that they correspond to the LDF defined herein.
\end{Remark}

\begin{Remark}\label{remark-2}
The $(72,4,1)$-LDF in Lemma $\ref{lem:72-LDF}$ has also been implicitly used in \textnormal{\cite{zzcfwz2024}} to establish the existence of cyclic relative difference families with block size four, in \textnormal{\cite{zcf}} to demonstrate the existence of optimal optical orthogonal codes with weight four achieving the Johnson bound, and in \textnormal{\cite{zcf2025}} to establish the existence of cyclic balanced sampling plans avoiding adjacent units with maximum distance five and block size four.
\end{Remark}
	
\begin{Proposition}\label{LDF-nec-1}
A $(v,k,\lambda)$-LDF consists of $\lambda v/(k(k-1))$ base blocks.
\end{Proposition}
	
\begin{proof}
The sum of the coefficients of all differences generated from each base block of a $(v,k,\lambda)$-LDF is
$$\sum_{0\le s_1< s_2 < k} \sum_{0\le \ell< s_2-s_1} \left(\frac{1}{s_2-s_1}\cdot 2\right)=\sum_{0\le s_1< s_2 < k}  2 =k(k-1).$$
Assume that a $(v,k,\lambda)$-LDF has $n$ base blocks. Then according to the definition of layered difference families, we have $nk(k-1)=\lambda v$.
\end{proof}

The $(72,4,1)$-LDF given in the proof of Lemma \ref{lem:72-LDF} cannot be used to construct $(v,4,1)$-PDFs. For instance, the ordered pair $(0,43)$ in the base block $B_1$ has $43$ as the positive difference, which exceeds $72/2 = 36$. This limitation motivates us to define perfect layered difference families.

\begin{Definition}
Let $k \geq 2$ be an integer and $v$ be an even positive integer. Let $B = (b_0, \ldots, b_{k-1})$ be an ordered $k$-multi-subset of $(\mathbb{Z}_v, +)$. For an ordered pair $(b_{s_1}, b_{s_2})$ with $0\leq s_1<s_2<k$, define
$$
\Delta^{*+} (b_{s_1}, b_{s_2}) :=
	\begin{cases}
		\sum\limits_{0 \leq \ell < s_2-s_1} \underline{\frac{1}{s_2-s_1}}(b_{s_2} - b_{s_1} + \ell), & \text{if } b_{s_1} \ (\mathrm{mod}\ v) \leq b_{s_2} \ (\mathrm{mod}\ v); \\
		\sum\limits_{0 \leq \ell < s_2-s_1} \underline{\frac{1}{s_2-s_1}}(b_{s_1} - b_{s_2} - \ell - 1), & \text{if } b_{s_1} \ (\mathrm{mod}\ v) > b_{s_2} \ (\mathrm{mod}\ v).
	\end{cases}
	$$
Let
$$
\Delta^{*+}B := \sum_{0 \leq s_1 < s_2 < k} \Delta^{*+} (b_{s_1}, b_{s_2}).
$$
Let $\B$ be a collection of ordered $k$-multi-subsets $($called \emph{base blocks}$)$ of $\Z_v$. If
$$
\Delta^{*+}\mathcal{B}  := \sum_{B \in \mathcal{B}} \Delta^{*+}B = \underline{\lambda} \sum_{g = 0}^{v/2 - 1} g\in \mathbb{Q}[\Z_v],
$$
then $\mathcal{B}$ is called a \emph{perfect $(v,k,\lambda)$-layered difference family}, written as a $(v,k,\lambda)$-PLDF.
\end{Definition}

Note that if $\B$ is a $(v,k,\lambda)$-LDF with odd $v$, then the difference $(v-1)/2$ corresponds to the interval $[(v-1)t/2, (v+1)t/2]$ in its associated cyclic difference packing  (recall that we denote an interval by the coefficient of $t$ in its starting point; see Table \ref{tab:auxtab}). Such an interval cannot be properly accommodated in a perfect difference family. Consequently, when we define a perfect layered difference family, we require that $v$ is even.

Computer search shows that no $(72,4,1)$-PLDF exists. Therefore, in order to construct $(v,4,\lambda)$-PDFs or $(m,4,3)$-PSDSs in next sections, we construct a $(108,4,1)$-PLDF as follows.


\begin{Lemma}\label{lem:108-LDF}
There exists a $(108,4,1)$-PLDF.
\end{Lemma}

\begin{proof}
Take the following ordered 4-multi-subsets of $\Z_{108}$:
\begin{center}
\begin{tabular}{llll}
$B_{1}=(0,18,42,0)$, & $B_{2}=(0,50,30,0)$, & $B_{3}=(0,53,45,0)$, \\
$B_{4}=(0,34,25,20)$, & $B_{5}=(0,38,48,20)$, & $B_{6}=(0,47,32,20)$, \\
$B_{7}=(0,3,42,35)$, & $B_{8}=(0,5,45,35)$, & $B_{9}=(0,23,51,35)$.
\end{tabular}
\end{center}
One can check that $\sum_{j=1}^{9} \Delta^{*+} B_{j}=\sum_{g=0}^{53} g$, and so $B_{1},B_{2},\dots,B_{9}$ form a $(108,4,1)$-PLDF.
\end{proof}

\section{Existence of $(m,4,3)$-PSDSs}\label{sec:main-construction}

In this section, we utilize the $(108,4,1)$-PLDF from Lemma \ref{lem:108-LDF} to explicitly construct $(m,4,3)$-PSDSs for all $m \geq 5$, thereby proving Theorem \ref{thm:(m,4,3)-PSDS}.

\begin{Lemma}\label{lem:small-PSDS}
There exists an $(m,4,3)$-PSDS for any integer $5\leq m\leq 16$.
\end{Lemma}

\begin{proof}
We provide a direct construction for an $(m,4,3)$-PSDS with $5\leq m\leq 16$ in Appendix \ref{Appendix:small-PSDS}.
\end{proof}

\begin{Lemma}\label{lem:m=8(mod9)}
There exists an $(m,4,3)$-PSDS for any positive integer $m \equiv 8 \pmod{9}$.
\end{Lemma}

\begin{proof}
For $m=8$, an $(m,4,3)$-PSDS exists by Lemma \ref{lem:small-PSDS}. For $m\geq 17$, let $m=9t-1$, where $t\geq 2$, the first $9t-18$ base blocks are listed below:
	\begin{center}
		\begin{tabular}{llll}
			$F_{1,i}:=\{0$, & $18t+i-2$, & $42t+2i-5$, & $3i\}$, \\
			$F_{2,i}:=\{0$, & $50t+i-5$, & $30t+2i-3$, & $3i+1\}$, \\
			$F_{3,i}:=\{0$, & $53t+i-5$, & $45t+2i-3$, & $3i+2\}$, \\
			$F_{4,i}:=\{0$, & $34t+i-4$, & $25t+2i-2$, & $20t+3i-2\}$, \\
			$F_{5,i}:=\{0$, & $38t+i-4$, & $48t+2i-5$, & $20t+3i-1\}$, \\
			$F_{6,i}:=\{0$, & $47t+i-6$, & $32t+2i-4$, & $20t+3i-3\}$, \\
			$F_{7,i}:=\{0$, & $3t+i$, & $42t+2i-4$, & $35t+3i-3\}$, \\
			$F_{8,i}:=\{0$, & $5t+i$, & $45t+2i-4$, & $35t+3i-2\}$, \\
			$F_{9,i}:=\{0$, & $23t+i-3$, & $51t+2i-6$, & $35t+3i-4\}$, \\
\end{tabular}
\end{center}
where $1\leq i\leq t-2$. The remaining 17 base blocks are listed below:
\begin{center}
\begin{tabular}{lll}
			$\{0, 4t-1, 15t-1, 38t-7\}$,
			&$\{0, 14t-3, 30t-3, 53t-8\}$,
			&$\{0, 3t-3, 38t-5, 48t-7\}$,
			\\
			$\{0, 4t+1, 45t-3, 54t-4\}$,
			&$\{0, 3t-2, 19t-3, 44t-7\}$,
			&$\{0, 5t, 25t-1, 53t-5\}$,
			\\
			$\{0, 4t, 42t-4, 48t-4\}$,
			&$\{0, 11t-1, 30t-2, 53t-6\}$,
			&$\{0, 3t, 35t-3, 54t-5\}$,
			\\
			$\{0, 14t-2, 29t-4, 54t-6\}$,
			&$\{0, 11t-2, 38t-6, 50t-6\}$,
			&$\{0, 8t-2, 35t-5, 53t-7\}$,
			\\
			$\{0, 3t-1, 32t-4, 50t-7\}$,
			&$\{0, 12t-1, 20t-2, 44t-6\}$,
			&$\{0, 14t-1, 23t-3, 48t-6\}$,
			\\
			$\{0, 7t-1, 35t-4, 51t-6\}$,
			&$\{0, 6t-1, 30t-4, 40t-5\}$.
			&
\end{tabular}
\end{center}

\begin{table}[!b]
\resizebox{\linewidth}{!}{
\begin{tabular}{|c|c||c|c||c|c|}\hline
		$\Delta F^+_{1,i}$ & $\Delta{\cal F}^+_1$ &  $\Delta F^+_{2,i}$ & $\Delta{\cal F}^+_2$ & $\Delta F^+_{3,i}$ & $\Delta{\cal F}^+_3$  \\
		\hline
		$18t+i-2$ & $[18t-1,19t-4]$ & $50t+i-5$ & $[50t-4,51t-7]$ & $53t+i-5$ & $[53t-4,54t-7]$\\
		$24t+i-3$ & $[24t-2,25t-5]$ & $20t-i-2$ & $[19t,20t-3]$ & $8t-i-2$ & $[7t,8t-3]$ \\
		$42t-i-5$ & $[41t-3,42t-6]$ & $30t-i-4$ & $[29t-2,30t-5]$ & $45t-i-5$ & $[44t-3,45t-6]$ \\
		$42t+2i-5$ & $[42t-3,44t-9]_{2}$ & $30t+2i-3$ & $[30t-1,32t-7]_{2}$ & $45t+2i-3$ & $[45t-1,47t-7]_{2}$ \\
		$18t-2i-2$ & $[16t+2,18t-4]_{2}$ & $50t-2i-6$ & $[48t-2,50t-8]_{2}$ & $53t-2i-7$ & $[51t-3,53t-9]_{2}$ \\
		$3i$ & $[3,3t-6]_{3}$ & $3i+1$ & $[4,3t-5]_{3}$ & $3i+2$ & $[5,3t-4]_{3}$ \\
		\hline
$\Delta F^+_{4,i}$ & $\Delta{\cal F}^+_4$ &  $\Delta F^+_{5,i}$ & $\Delta{\cal F}^+_5$ & $\Delta F^+_{6,i}$ & $\Delta{\cal F}^+_6$  \\\hline
		$34t+i-4$ & $[34t-3,35t-6]$ & $38t+i-4$ & $[38t-3,39t-6]$ & $47t+i-6$ & $[47t-5,48t-8]$ \\
		$9t-i-2$ & $[8t,9t-3]$ & $10t+i-1$ & $[10t,11t-3]$ & $15t-i-2$ & $[14t,15t-3]$ \\
		$5t-i$ & $[4t+2,5t-1]$ & $28t-i-4$ & $[27t-2,28t-5]$ & $12t-i-1$ & $[11t+1,12t-2]$ \\
		$25t+2i-2$ & $[25t,27t-6]_{2}$ & $48t+2i-5$ & $[48t-3,50t-9]_{2}$ & $32t+2i-4$ & $[32t-2,34t-8]_{2}$ \\
		$14t-2i-2$ & $[12t+2,14t-4]_{2}$ & $18t-2i-3$ & $[16t+1,18t-5]_{2}$ & $27t-2i-3$ & $[25t+1,27t-5]_{2}$ \\
		$20t+3i-2$ & $[20t+1,23t-8]_{3}$ & $20t+3i-1$ & $[20t+2,23t-7]_{3}$ & $20t+3i-3$ & $[20t,23t-9]_{3}$ \\
		\hline
		$\Delta F^+_{7,i}$ & $\Delta{\cal F}^+_7$ &  $\Delta F^+_{8,i}$ & $\Delta{\cal F}^+_8$ & $\Delta F^+_{9,i}$ & $\Delta{\cal F}^+_9$  \\
		\hline
		$3t+i$ & $[3t+1,4t-2]$ & $5t+i$ & $[5t+1,6t-2]$ & $23t+i-3$ & $[23t-2,24t-5]$ \\
		$39t+i-4$ & $[39t-3,40t-6]$ & $40t+i-4$ & $[40t-3,41t-6]$ & $28t+i-3$ & $[28t-2,29t-5]$ \\
		$7t-i-1$ & $[6t+1,7t-2]$ & $10t-i-2$ & $[9t,10t-3]$ & $16t-i-2$ & $[15t,16t-3]$ \\
		$42t+2i-4$ & $[42t-2,44t-8]_{2}$ & $45t+2i-4$ & $[45t-2,47t-8]_{2}$ & $51t+2i-6$ & $[51t-4,53t-10]_{2}$ \\
		$32t+2i-3$ & $[32t-1,34t-7]_{2}$ & $30t+2i-2$ & $[30t,32t-6]_{2}$ & $12t+2i-1$ & $[12t+1,14t-5]_{2}$ \\
		$35t+3i-3$ & $[35t,38t-9]_{3}$ & $35t+3i-2$ & $[35t+1,38t-8]_{3}$ &$35t+3i-4$ & $[35t-1,38t-10]_{3}$ \\
		\hline
\end{tabular}}
\caption{Positive differences from the first $9t-18$ base blocks in Lemma \ref{lem:m=8(mod9)}}\label{tab:diff-PSDS}
\end{table}

Note that every $F_{r,i}$ with $1\leq r\leq 9$ is of the form
\begin{align*}
F_{r,i}=\{0,\ \ \ b_{r,1}t+i+\epsilon_{r,1},\ \ \ b_{r,2}t+2 i+\epsilon_{r,2}, \ \ \ b_{r,3} t+3 i+\epsilon_{r,3}\},
\end{align*}
and we require that the ordered 4-multisets $(0,b_{r,1},b_{r,2},b_{r,3})$, $1\leq r\leq 9$, form a $(108,4,1)$-PLDF as given in Lemma \ref{lem:108-LDF}. After fixing these $b_{r,1}$, $b_{r,2}$ and $b_{r,3}$, we perform a computer search for suitable values of $\epsilon_{r,1}$, $\epsilon_{r,2}$ and $\epsilon_{r,3}$ such that the differences generated by these $F_{r,i}$ with $1\leq r\leq 9$ are pairwise distinct, where $\epsilon_{r,1}$, $\epsilon_{r,2}$ and $\epsilon_{r,3}$ are restricted to integers in $[-6,6]$. Once this is successful, we then use a computer search to find the remaining 17 base blocks by employing the Kramer--Mesner method, one of the most common methods for searching designs (see \cite[Section 9.2.1]{ko}). If the search fails, we adjust the values of $\epsilon_{r,1}$, $\epsilon_{r,2}$, $\epsilon_{r,3}$, and continue the process. Ultimately, the base blocks listed above are obtained.

To facilitate the reader to check the correctness of our results, we list all positive differences from $F_{r,i}$ with $1\leq r\leq 9$ in Table \ref{tab:diff-PSDS}, where $\F_r=\{F_{r,i}: 1\leq i\leq{t-2}\}$. We list all positive differences from the remaining 17 base blocks in Table \ref{tab:diff-base-blocks-remaining}.
\end{proof}

\begin{table}[t!]
\renewcommand{\arraystretch}{1.2}
\resizebox{\linewidth}{!}{
\begin{tabular}{|c|c|}
\hline
$\{4t-1,15t-1,38t-7,11t,34t-6,23t-6\}$ &
$\{14t-3,30t-3,53t-8,16t,39t-5,23t-5\}$ \\
[0.25ex]\hline
$\{3t-3,38t-5,48t-7,35t-2,45t-4,10t-2\}$ &
$\{4t+1,45t-3,54t-4,41t-4,50t-5,9t-1\}$ \\
[0.25ex]\hline
$\{3t-2,19t-3,44t-7,16t-1,41t-5,25t-4\}$ &
$\{5t,25t-1,53t-5,20t-1,48t-5,28t-4\}$ \\
[0.25ex]\hline
$\{4t,42t-4,48t-4,38t-4,44t-4,6t\}$ &
$\{11t-1,30t-2,53t-6,19t-1,42t-5,23t-4\}$ \\
[0.25ex]\hline
$\{3t,35t-3,54t-5,32t-3,51t-5,19t-2\}$ &
$\{14t-2,29t-4,54t-6,15t-2,40t-4,25t-2\}$ \\
[0.25ex]\hline
$\{11t-2,38t-6,50t-6,27t-4,39t-4,12t\}$ &
$\{8t-2,35t-5,53t-7,27t-3,45t-5,18t-2\}$ \\
[0.25ex]\hline
$\{3t-1,32t-4,50t-7,29t-3,47t-6,18t-3\}$ &
$\{12t-1,20t-2,44t-6,8t-1,32t-5,24t-4\}$ \\
[0.25ex]\hline
$\{14t-1,23t-3,48t-6,9t-2,34t-5,25t-3\}$ &
$\{7t-1,35t-4,51t-6,28t-3,44t-5,16t-2\}$ \\
[0.25ex]\hline
$\{6t-1,30t-4,40t-5,24t-3,34t-4,10t-1\}$ & \\
[0.25ex]\hline
\end{tabular}
}
\caption{Positive differences from the remaining 17 base blocks}
\label{tab:diff-base-blocks-remaining}
\end{table}

By a similar method used in the proof of Lemma \ref{lem:m=8(mod9)}, we can construct an $(m,4,3)$-PSDS for every positive integer $m\geq 5$ and $m \not\equiv 8 \pmod{9}$.

\begin{Lemma}\label{lem:other-PSDS}
There exists an $(m,4,3)$-PSDS for any positive integer $m \geq 5$ and $m \not\equiv 8 \pmod{9}$.
\end{Lemma}

\begin{proof}
For $5\leq m\leq 16$ and $m\neq 8$, an $(m,4,3)$-PSDS exists by Lemma \ref{lem:small-PSDS}. For $m \geq 18$, let $m = 9t+x$, where $t\geq 2$ and $0 \leq x \leq 7$. The first $9t-18$ base blocks are listed below:
	\begin{center}
		\begin{tabular}{llll}
			$F_{1,i}:=\{0$, & $18t+i+a_1$, & $42t+2i+b_1$, & $3i+c_1\}$, \\
			$F_{2,i}:=\{0$, & $50t+i+a_2$, & $30t+2i+b_2$, & $3i+c_2\}$, \\
			$F_{3,i}:=\{0$, & $53t+i+a_3$, & $45t+2i+b_3$, & $3i+c_3\}$, \\
			$F_{4,i}:=\{0$, & $34t+i+a_4$, & $25t+2i+b_4$, & $20t+3i+c_4\}$, \\
			$F_{5,i}:=\{0$, & $38t+i+a_5$, & $48t+2i+b_5$, & $20t+3i+c_5\}$, \\
			$F_{6,i}:=\{0$, & $47t+i+a_6$, & $32t+2i+b_6$, & $20t+3i+c_6\}$, \\
			$F_{7,i}:=\{0$, & $3t+i+a_7$, & $42t+2i+b_7$, & $35t+3i+c_7\}$, \\
			$F_{8,i}:=\{0$, & $5t+i+a_8$, & $45t+2i+b_8$, & $35t+3i+c_8\}$, \\
			$F_{9,i}:=\{0$, & $23t+i+a_9$, & $51t+2i+b_9$, & $35t+3i+c_9\}$, \\
		\end{tabular}
	\end{center}
where $1\leq i\leq t-2$, and $a_j, b_j, c_j$ for $1 \leq j \leq 9$ are listed in Table \ref{tab:4.3} according to the value of $x$. The remaining $18+x$ base blocks are provided in Appendix \ref{Appendix:main}.

For each case of $x$, the search for the parameters $a_j, b_j, c_j$ and the sporadic base blocks took between half a day and two days on a standard personal computer.

To help readers verify our results, the Python script is available at \cite{Checking}.
\end{proof}

\begin{table}[!t]\setlength{\tabcolsep}{2.9pt}
\begin{tabularx}{\textwidth}{|c|ccc|ccc|ccc|ccc|ccc|ccc|ccc|ccc|ccc|}\hline
			$x$ & $a_1$ & $b_1$ & $c_1$ & $a_2$ & $b_2$ & $c_2$ & $a_3$ & $b_3$ & $c_3$ & $a_4$ & $b_4$ & $c_4$ & $a_5$ & $b_5$ & $c_5$ & $a_6$ & $b_6$ & $c_6$ & $a_7$ & $b_7$ & $c_7$ & $a_8$ & $b_8$ & $c_8$ & $a_9$ & $b_9$ & $c_9$\\
			\hline
			$0$ & $0$ & $0$ & $0$ & $0$ & $1$ & $1$ & $0$ & $1$ & $2$ & $1$ & $1$ & $2$ & $0$ & $0$ & $1$ & $0$ & $0$ & $0$ & $0$ & $1$ & $1$ & $1$ & $2$ & $3$ & $0$ & $1$ & $2$\\
			
			$1$ & $2$ & $4$ & $0$ & $5$ & $3$ & $1$ & $5$ & $5$ & $2$ & $4$ & $4$ & $4$ & $4$ & $5$ & $3$ & $5$ & $4$ & $2$ & $0$ & $5$ & $3$ & $1$ & $6$ & $5$ & $3$ & $6$ & $4$\\
			
			$2$ & $4$ & $9$ & $0$ & $11$ & $6$ & $1$ & $12$ & $11$ & $2$ & $7$ & $7$ & $6$ & $8$ & $11$ & $5$ & $10$ & $7$ & $4$ & $1$ & $10$ & $9$ & $1$ & $10$ & $8$ & $5$ & $11$ & $7$\\
			
			$3$ & $6$ & $14$ & $0$ & $16$ & $10$ & $1$ & $18$ & $16$ & $2$ & $12$ & $10$ & $8$ & $12$ & $16$ & $7$ & $15$ & $11$ & $6$ & $1$ & $15$ & $11$ & $2$ & $15$ & $13$ & $7$ & $17$ & $12$\\
			
			$4$ & $8$ & $18$ & $0$ & $22$ & $13$ & $1$ & $23$ & $20$ & $2$ & $15$ & $11$ & $9$ & $17$ & $22$ & $10$ & $20$ & $14$ & $8$ & $1$ & $19$ & $16$ & $3$ & $21$ & $17$ & $10$ & $22$ & $15$\\
			
			$5$ & $10$ & $23$ & $0$ & $27$ & $16$ & $1$ & $29$ & $25$ & $2$ & $18$ & $14$ & $11$ & $21$ & $27$ & $12$ & $26$ & $19$ & $13$ & $2$ & $24$ & $20$ & $4$ & $26$ & $21$ & $13$ & $28$ & $19$\\
			
			$6$ & $12$ & $28$ & $0$ & $33$ & $20$ & $1$ & $35$ & $30$ & $2$ & $22$ & $17$ & $15$ & $25$ & $33$ & $14$ & $31$ & $21$ & $13$ & $2$ & $29$ & $24$ & $4$ & $31$ & $25$ & $15$ & $34$ & $23$\\
			
			$7$ & $14$ & $32$ & $0$ & $38$ & $23$ & $1$ & $41$ & $35$ & $2$ & $26$ & $20$ & $17$ & $29$ & $38$ & $16$ & $36$ & $25$ & $15$ & $2$ & $33$ & $28$ & $5$ & $36$ & $29$ & $17$ & $40$ & $27$\\
			\hline
\end{tabularx}
\caption{Parameters $a_j, b_j, c_j$ in Lemma \ref{lem:other-PSDS}}\label{tab:4.3}
\end{table}

\begin{proof}[Proof of Theorem \ref{thm:(m,4,3)-PSDS}]
One can combine Lemmas \ref{lem:m=8(mod9)} and \ref{lem:other-PSDS} to obtain the sufficiency of Theorem \ref{thm:(m,4,3)-PSDS}. The necessity comes from Proposition \ref{thm:PSDS-nece}(2).
\end{proof}

\section{Existence of $(v,4,\lambda)$-PDFs}\label{sec:multi}

This section establishes the existence of $(v,4,\lambda)$-PDFs for every $\lambda\geq 1$. We provide Python code \cite{Checking} to facilitate verification of our results for the cases $\lambda=2,3,6$.

\subsection{$\lambda=1$}\label{sec:multi-0}

Now we provide an affirmative answer to Conjecture \ref{conj:Berm}.

\begin{Lemma}\label{thm:4-PDF}
A $(v,4,1)$-PDF exists if and only if $v\equiv 1\pmod{12}$ and $v\neq 25,37$.
\end{Lemma}

\begin{proof}
A $(v,4,1)$-PDF consists of $(v-1)/12$ base blocks, so $v\equiv 1\pmod{12}$. By Lemma \ref{lem:k=4-small orders}, there is no $(v,4,1)$-PDF for $v\in\{25,37\}$. On the other hand, by Theorem \ref{thm:(m,4,3)-PSDS} and Lemma \ref{lem:c=3 to 1}, there exists an $(m+1,4,1)$-PSDS for every $m\geq 5$, which implies a $(12m+13,4,1)$-PDF. Combine Lemma \ref{lem:k=4-small orders} to complete the proof.
\end{proof}

\subsection{$\lambda=2$}\label{sec:multi-1}

\begin{Lemma}\label{lem:36-4-2}
There exists a $(36,4,2)$-PLDF.
\end{Lemma}

\begin{proof}
Take the following ordered 4-multi-subsets of $\Z_{36}$:
	\begin{center}
		\begin{tabular}{lll}
			$B_{1}=(0,0,3,15)$, &
			$B_{2}=(0,0,6,15)$, &
			$B_{3}=(0,11,1,15)$,\\
			$B_{4}=(0,10,1,11)$,&
			$B_{5}=(0,14,6,11)$,&
			$B_{6}=(0,17,3,11)$.\\
		\end{tabular}
	\end{center}
One can verify that $\sum_{j=1}^{6} \Delta^{*+} B_{j}=\underline{2}\sum_{g=0}^{17} g$. Then $B_{1},B_{2},\dots,B_{6}$ form a $(36,4,2)$-PLDF.
\end{proof}

\begin{Lemma}\label{thm:(v,4,2)PDF}
There exists a $(v,4,2)$-PDF for any $v \equiv 1 \pmod{6}$ and $v\geq 13$.
\end{Lemma}

\begin{proof}
For $v\equiv 1 \pmod{12}$ and $v\neq 25,37$, by Lemma \ref{thm:4-PDF}, there exists a $(v,4,1)$-PDF. Taking two copies of a $(v,4,1)$-PDF yields a $(v,4,2)$-PDF. For $v\in\{25,37\}\cup \{19, 31, 43, 55\}$, see Appendix \ref{Appendix:small-PDF-2}.

For $v\equiv 7 \pmod{12}$ and $v\geq 67$, let $v = 36t + 6x + 1$, where $t \geq 2$ and $x \in \{-1, 1, 3\}$. A $(v,4,2)$-PDF consists of $6t+x$ base blocks. The first $6t-12$ base blocks are listed below:
	\begin{center}		
		\begin{tabular}{llll}
			$F_{1,i}:=\{0$, & $i+a_1$, & $3t+2i+b_1$, & $15t+3i+c_1\}$, \\
			$F_{2,i}:=\{0$, & $i+a_2$, & $6t+2i+b_2$, & $15t+3i+c_2\}$, \\
			$F_{3,i}:=\{0$, & $11t+i+a_3$, & $t+2i+b_3$, & $15t+3i+c_3\}$, \\
			$F_{4,i}:=\{0$, & $10t+i+a_4$, & $t+2i+b_4$, & $11t+3i+c_4\}$, \\
			$F_{5,i}:=\{0$, & $14t+i+a_5$, & $6t+2i+b_5$, & $11t+3i+c_5\}$, \\
			$F_{6,i}:=\{0$, & $17t+i+a_6$, & $3t+2i+b_6$, & $11t+3i+c_6\}$, \\
		\end{tabular}
	\end{center}
where $1\leq i\leq t-2$, and $a_j,b_j,c_j$ for $1 \leq j \leq 6$ are listed in Table \ref{tab:5.3} according to the value of $x$. The remaining $12+x$ base blocks are provided in Appendix \ref{Appendix:big-PDF-2}.

Note that the coefficients of $t$ in the first $6t-12$ base blocks form a $(36,4,2)$-PLDF as given in Lemma \ref{lem:36-4-2}. For each $x$, the search for the parameters $a_j, b_j, c_j$ and the sporadic base blocks took under an hour on a standard personal computer. Analogous results were obtained for $\lambda = 3$ and $6$, and will not be repeated hereafter.
\end{proof}

\begin{table}[htb]\centering
\begin{tabular}{|c|ccc|ccc|ccc|ccc|ccc|ccc|}\hline
$x$ & $a_1$ & $b_1$ & $c_1$ & $a_2$ & $b_2$ & $c_2$ & $a_3$ & $b_3$ & $c_3$ & $a_4$ & $b_4$ & $c_4$ & $a_5$ & $b_5$ & $c_5$ & $a_6$ & $b_6$ & $c_6$\\\hline
			$-1$ & $1$ & $0$ & $0$ & $0$ & $-1$ & $-2$ & $-1$ & $0$ & $-1$ & $1$ & $1$ & $-1$ & $0$ & $0$ & $1$ & $-1$ & $1$ & $0$\\
			$1$ & $0$ & $0$ & $3$ & $2$ & $2$ & $4$ & $1$ & $1$ & $2$ & $3$ & $1$ & $1$ & $3$ & $1$ & $3$ & $4$ & $1$ & $2$\\
			$3$ & $1$ & $1$ & $7$ & $2$ & $4$ & $9$ & $5$ & $2$ & $8$ & $6$ & $2$ & $7$ & $7$ & $3$ & $6$ & $9$ & $2$ & $5$\\\hline
\end{tabular}
\caption{Parameters $a_j, b_j, c_j$ in Lemma \ref{thm:(v,4,2)PDF}}\label{tab:5.3}
\end{table}

\subsection{$\lambda=3$}\label{sec:multi-2}

\begin{Lemma}\label{lem:24-4-3}
There exists a $(24,4,3)$-PLDF.
\end{Lemma}

\begin{proof}
Take the following ordered 4-multi-subsets of $\Z_{24}$:
	\begin{center}
		\begin{tabular}{lll}
			$B_{1}=(0,0,8,0)$,&
			$B_{2}=(0,8,10,0)$,&
			$B_{3}=(0,10,4,0)$,\\
			$B_{4}=(0,9,10,3)$,&
			$B_{5}=(0,11,0,3)$,&
			$B_{6}=(0,11,6,3)$.\\
		\end{tabular}
	\end{center}
One can check that $\sum_{j=1}^{6} \Delta^{*+} B_{j}=\underline{3}\sum_{g=0}^{11} g$. Then $B_{1},B_{2},\dots,B_{6}$ form a $(24,4,3)$-PLDF.	
\end{proof}

\begin{Lemma}\label{thm:(v,4,3)PDF}
There exists a $(v,4,3)$-PDF for any $v \equiv 1 \pmod{4}$ and $v\geq 13$.
\end{Lemma}

\begin{proof}
For $v\equiv 1 \pmod{12}$ and $v\neq 25,37$, taking three copies of a $(v,4,1)$-PDF from Lemma \ref{thm:4-PDF} yields a $(v,4,3)$-PDF. For $v\in\{25,37\}\cup \{17, 21, 29, 33\}$, see Appendix \ref{Appendix:small-PDF-3}.

For $v \equiv 5, 9 \pmod{12}$ and $v\geq 41$, let $v = 24t + 4x + 1$, where $t \geq 2$ and $x \in \{-2, -1, 1, 2\}$. A $(v,4,3)$-PDF consists of $6t+x$ base blocks. The first $6t - 12$ base blocks are listed below:
\begin{center}		
		\begin{tabular}{llll}
			$F_{1,i}:=\{0$, & $i+a_1$, & $8t+2i+b_1$, & $3i+c_1\}$, \\
			$F_{2,i}:=\{0$, & $8t+i+a_2$, & $10t+2i+b_2$, & $3i+c_2\}$, \\
			$F_{3,i}:=\{0$, & $10t+i+a_3$, & $4t+2i+b_3$, & $3i+c_3\}$, \\
			$F_{4,i}:=\{0$, & $9t+i+a_4$, & $10t+2i+b_4$, & $3t+3i+c_4\}$, \\
			$F_{5,i}:=\{0$, & $11t+i+a_5$, & $2i+b_5$, & $3t+3i+c_5\}$, \\
			$F_{6,i}:=\{0$, & $11t+i+a_6$, & $6t+2i+b_6$, & $3t+3i+c_6\}$, \\
		\end{tabular}
	\end{center}
where $1\leq i\leq t-2$, and $a_j,b_j,c_j$ for $1 \leq j \leq 6$ are listed in Table \ref{tab:5.4} according to the value of $x$. The remaining $12+x$ base blocks are given in Appendix \ref{Appendix:big-PDF-3}. Note that the coefficients of $t$ in the first $6t-12$ base blocks form a $(24,4,3)$-PLDF as given in Lemma \ref{lem:24-4-3}.	
\end{proof}

\begin{table}[htpb]\centering
		\begin{tabular}{|c|ccc|ccc|ccc|ccc|ccc|ccc|}\hline
			$x$ & $a_1$ & $b_1$ & $c_1$ & $a_2$ & $b_2$ & $c_2$ & $a_3$ & $b_3$ & $c_3$ & $a_4$ & $b_4$ & $c_4$ & $a_5$ & $b_5$ & $c_5$ & $a_6$ & $b_6$ & $c_6$\\
			\hline
			$-2$ & $1$ & $-1$ & $0$ & $-2$ & $-3$ & $2$ & $-1$ & $0$ & $1$ & $-2$ & $-2$ & $1$ & $-2$ & $2$ & $-1$ & $-3$ & $0$ & $0$\\
			$-1$ & $0$ & $0$ & $0$ & $-1$ & $0$ & $1$ & $1$ & $0$ & $2$ & $0$ & $1$ & $1$ & $-1$ & $1$ & $2$ & $0$ & $1$ & $0$\\
			$1$ & $0$ & $3$ & $0$ & $1$ & $2$ & $2$ & $3$ & $0$ & $1$ & $1$ & $3$ & $2$ & $1$ & $1$ & $1$ & $2$ & $1$ & $0$\\
			$2$ & $0$ & $4$ & $0$ & $3$ & $3$ & $2$ & $4$ & $1$ & $1$ & $3$ & $4$ & $3$ & $5$ & $1$ & $1$ & $3$ & $4$ & $2$\\ \hline
		\end{tabular}
		\caption{Parameters $a_j, b_j, c_j$ in Lemma \ref{thm:(v,4,3)PDF}}\label{tab:5.4}
	\end{table}

\subsection{$\lambda=6$}\label{sec:multi-3}

\begin{Lemma}\label{lem:12-4-6}
There exists a $(12,4,6)$-PLDF.
\end{Lemma}

\begin{proof}
Take the following ordered 4-multi-subsets of $\Z_{12}$:
	\begin{center}
		\begin{tabular}{lll}
			$B_{1}=(0,2,3,0)$,&
			$B_{2}=(0,5,2,0)$,&
			$B_{3}=(0,5,4,0)$,\\
			$B_{4}=(0,0,4,3)$,&
			$B_{5}=(0,5,0,3)$,&
			$B_{6}=(0,5,2,3)$.\\
		\end{tabular}
	\end{center}
One can check that $\sum_{j=1}^{6} \Delta^{*+} B_{j}=\underline{6}\sum_{g=0}^5 g$. Then $B_{1},B_{2},\dots,B_{6}$ form a $(12,4,6)$-PLDF.
\end{proof}

\begin{Lemma}\label{thm:(v,4,6)PDF}
There exists a $(v,4,6)$-PDF for any $v \equiv 1 \pmod{2}$ and $v\geq 13$.
\end{Lemma}

\begin{proof}
For $v\equiv 1 \pmod{4}$ and $v\geq 13$, taking two copies of a $(v,4,3)$-PDF from Lemma \ref{thm:(v,4,3)PDF} yields a $(v,4,6)$-PDF. For $v\equiv 7 \pmod{12}$ and $v\geq 19$, taking three copies of a $(v,4,2)$-PDF from Lemma \ref{thm:(v,4,2)PDF} yields a $(v,4,6)$-PDF. For $v=15$, a $(15,4,6)$-PDF is given below:
\begin{center}
\begin{tabular}{lllllllll}
$\{0,1,4,6\}$,& $\{0,2,5,7\}$,& $\{0,2,6,7\}$,& $\{0,2,6,7\}$,& $\{0,2,6,7\}$, & $\{0,3,6,7\}$,& $\{0,3,6,7\}$. \\
\end{tabular}
\end{center}

For $v \equiv 3, 11 \pmod{12}$ and $v\geq 23$, let $v = 12t + 2x + 1$, where $t \geq 2$ and $x \in \{-1, 1\}$. A $(v,4,2)$-PDF consists of $6t+x$ base blocks. The first $6t-12$ base blocks are listed below:
	\begin{center}		
		\begin{tabular}{llll}
			$F_{1,i}:=\{0$, & $2t+i+a_1$, & $3t+2i+b_1$, & $3i+c_1\}$, \\
			$F_{2,i}:=\{0$, & $5t+i+a_2$, & $2t+2i+b_2$, & $3i+c_2\}$, \\
			$F_{3,i}:=\{0$, & $5t+i+a_3$, & $4t+2i+b_3$, & $3i+c_3\}$, \\
			$F_{4,i}:=\{0$, & $i+a_4$, & $4t+2i+b_4$, & $3t+3i+c_4\}$, \\
			$F_{5,i}:=\{0$, & $5t+i+a_5$, & $2i+b_5$, & $3t+3i+c_5\}$, \\
			$F_{6,i}:=\{0$, & $5t+i+a_6$, & $2t+2i+b_6$, & $3t+3i+c_6\}$, \\
		\end{tabular}
	\end{center}
where $1\leq i\leq t-2$, and $a_j,b_j,c_j$ for $1 \leq j \leq 6$ are listed in Table \ref{tab:5.7} according to the value of $x$. The remaining $12+x$ base blocks are given in Appendix \ref{Appendix:big-PDF-6}. Note that the coefficients of $t$ in the first $6t-12$ base blocks form a $(12,4,6)$-PLDF as given in Lemma \ref{lem:12-4-6}.
\end{proof}

\subsection{General $\lambda$}\label{sec:multi-4}

We are now in a position to prove Theorem \ref{thm:(v,4,lambda)PDF}.

\begin{proof}[\bf{Proof of Theorem \ref{thm:(v,4,lambda)PDF}}]
A $(v,4,\lambda)$-PDF requires $\lambda(v-1)/12$ base blocks, implying $\lambda(v-1)\equiv 0\pmod{12}$. By Lemma \ref{lem:k=4-small orders}, there is no $(v,4,1)$-PDF for $v\in\{25,37\}$. For a $(v,4,\lambda)$-PDF, each difference in $[1,(v-1)/2]$ appears exactly $\lambda$ times. Since the largest difference $(v-1)/2$ can occur at most once per base block, the number of base blocks must satisfy $\lambda(v-1)/12 \geq \lambda$, which simplifies to $v \geq 13$. These establish the necessary conditions.

To prove sufficiency, assume that $\lambda(v-1)\equiv 0\pmod{12}$, $v\equiv 1\pmod{2}$, $v\geq 13$ and $(v,\lambda)\not\in\{(25,1),(37,1)\}$.

When $\lambda \equiv 1,5 \pmod{6}$, we have $v \equiv 1 \pmod{12}$. By Lemma \ref{thm:4-PDF}, a $(v,4,1)$-PDF exists for any $v \equiv 1 \pmod{12}$ and $v\not\in\{25,37\}$. Then $\lambda$ copies of this PDF form a $(v,4,\lambda)$-PDF. For $v=25$ and $37$, a $(v,4,\lambda)$-PDF with $\lambda>1$ can be obtained by suitable combinations of the $(v,4,2)$-PDF and $(v,4,3)$-PDF provided by Lemmas \ref{thm:(v,4,2)PDF} and \ref{thm:(v,4,3)PDF}, respectively.

When $\lambda \equiv 2,4 \pmod{6}$, since $v\equiv 1\pmod{2}$, we have $v \equiv 1 \pmod{6}$. By Lemma \ref{thm:(v,4,2)PDF}, a $(v,4,2)$-PDF exists for any $v \equiv 1 \pmod{6}$ and $v\geq 13$. Then $\lambda/2$ copies of this PDF form a $(v,4,\lambda)$-PDF.

When $\lambda \equiv 3 \pmod{6}$, we have $v \equiv 1 \pmod{4}$. By Lemma \ref{thm:(v,4,3)PDF}, a $(v,4,3)$-PDF exists for any $v \equiv 1 \pmod{4}$ and $v\geq 13$. Then $\lambda/3$ copies of this PDF form a $(v,4,\lambda)$-PDF.

When $\lambda \equiv 0 \pmod{6}$, we have $v \equiv 1 \pmod{2}$. By Lemma \ref{thm:(v,4,6)PDF}, a $(v,4,6)$-PDF exists for any $v \equiv 1 \pmod{2}$ and $v\geq 13$. Then $\lambda/6$ copies of this PDF form a $(v,4,\lambda)$-PDF.
\end{proof}

\begin{table}[!t]\centering
\begin{tabular}{|c|ccc|ccc|ccc|ccc|ccc|ccc|}
			\hline
			$x$ & $a_1$ & $b_1$ & $c_1$ & $a_2$ & $b_2$ & $c_2$ & $a_3$ & $b_3$ & $c_3$ & $a_4$ & $b_4$ & $c_4$ & $a_5$ & $b_5$ & $c_5$ & $a_6$ & $b_6$ & $c_6$\\
			\hline
			$-1$ & $1$ & $1$ & $1$ & $0$ & $1$ & $0$ & $0$ & $0$ & $2$ & $1$ & $1$ & $2$ & $1$ & $2$ & $0$ & $0$ & $2$ & $1$\\
			$1$ & $2$ & $0$ & $0$ & $1$ & $2$ & $1$ & $1$ & $2$ & $2$ & $1$ & $1$ & $2$ & $1$ & $2$ & $0$ & $2$ & $1$ & $1$\\\hline
\end{tabular}
\caption{Parameters $a_j, b_j, c_j$ in Lemma \ref{thm:(v,4,6)PDF}}\label{tab:5.7}
\end{table}

By Theorem \ref{thm:(v,4,lambda)PDF}, we can give a quick proof for Theorem \ref{thm:4-CDF-lambda}.

\begin{proof}[\bf{Proof of Theorem \ref{thm:4-CDF-lambda}}]
Clearly a $(v,4,\lambda)$-CDF exists only if $\lambda(v-1)\equiv 0\pmod{12}$. By \cite[Theorem 4]{zfw2022}, there is no $(25,4,1)$-CDF but there exists a $(37,4,1)$-CDF. Assume that $\lambda(v-1) \equiv 0 \pmod{12}$, $v \geq 4$ and $(v,\lambda)\neq (25,1)$. A value $v$ satisfying these conditions is called admissible.

The cases for small $v$ are handled directly: for $v=4$ and $\lambda \equiv 0 \pmod{4}$, a $(4,4,\lambda)$-CDF is constructed by $\lambda/4$ copies of $\{0,1,2,3\}$; for $v=5$ and $\lambda \equiv 0 \pmod{3}$, a $(5,4,\lambda)$-CDF is constructed by $\lambda/3$ copies of $\{0,1,2,4\}$; for $v=7$ and $\lambda \equiv 0 \pmod{2}$, a $(7,4,\lambda)$-CDF is constructed by $\lambda/2$ copies of $\{0,1,2,4\}$; for $v=9$ and $\lambda \equiv 0 \pmod{3}$, a $(9,4,\lambda)$-CDF is constructed by $\lambda/3$ copies of $\{0,1,2,5\}$ and $\{0,1,3,7\}$.

For larger $v$, the proof relies on Theorem \ref{thm:(v,4,lambda)PDF}. When $\lambda \equiv 1 \pmod{2}$ or $\lambda \equiv 2, 10 \pmod{12}$, a $(v,4,\lambda)$-CDF exists for all admissible $v \geq 13$. For other values of $\lambda$ modulo 12, namely $\lambda \equiv 4,8 \pmod{12}$, $\lambda \equiv 6 \pmod{12}$, and $\lambda \equiv 0 \pmod{12}$, a \((2v-1,4,\lambda/2)\)-PDF exists for every admissible $v \geq 7$. In each of these cases, the base blocks of this PDF cover every element in \(\{1,2,\dots,v-1\}\) exactly \(\lambda\) times when only considering their positive differences, thereby forming the required \((v,4,\lambda)\)-CDF. This completes the proof.
\end{proof}

\section{Applications}\label{sec:application}

\subsection{Additive sequences of permutations}

\begin{Definition}
Let $m$ and $n$ be positive integers. Let $X=\{x_1,\ldots,x_n\}$ be a set of distinct integers. For $j=1,\ldots,m$, let $X^{(j)}=(x_1^{(j)}, \ldots, x_n^{(j)})$ be a permutation of $X$. The sequence $(X^{(1)}, X^{(2)}, \ldots, X^{(m)})$ is called an \emph{additive sequence of permutations} of degree $n$, denoted by an ASP$(m,n)$, if for every subsequence of consecutive permutations $ (X^{(j_1)}, X^{(j_1+1)}, \ldots, X^{(j_2)}) $ with $1 \leq j_1 \leq j_2 \leq m$, the vector-sum $X^{(j_1)}+X^{(j_1+1)}+\cdots+X^{(j_2)}$ is also a permutation of $X$. The set $X$ is called the \emph{basis} of the additive sequence.
\end{Definition}

For example, the sequence $(X^{(1)}, X^{(2)})$, where
\begin{align*}
\begin{tabular}{ccccccccccc}
$X^{(1)}=(-r$, & $-r+1$, & $\ldots$, & $-1$, & $0$, & $1$, & $\ldots$, & $r-1$, & $r)$, and \\
$X^{(2)}=(0$,  & $1$, & $\ldots$, & $r-1$, & $r$, & $-r$, & $\ldots$, & $-2$, & $-1)$,
\end{tabular}
\end{align*}
forms an ASP$(2,2r+1)$ for any nonnegative integer $r$ with the basis $[-r,r]$ (see \cite{kl78}).

\begin{Theorem}\label{thm:ASP-2} {\rm \cite{kl78}}
There exists an ASP$(2,n)$ with the basis $[-(n-1)/2,(n-1)/2]$ for any $n\equiv 1\pmod{2}$.
\end{Theorem}

Let $(X^{(1)}, X^{(2)})$ be an ASP$(2,n)$ with the basis $X=\{x_1,\ldots,x_n\}$. Since $X^{(1)}$, $X^{(2)}$ and $X^{(1)}+X^{(2)}$ are all permutations of $X$, we have $\sum_{\ell=1}^n x_{\ell}=0$. This fact leads to the following necessary condition for the existence of an ASP$(m,n)$ having an integer interval as its basis.

\begin{Proposition}\emph{\cite[Theorem 2]{kt1984}}
If the basis $X$ of an ASP$(m,n)$ is an integer interval, then $n$ is odd and $X=[-(n-1)/2,(n-1)/2]$.
\end{Proposition}

The existence of ASP$(3,n)$s remains an open problem. Wang and Chang \cite[Theorem 2.10]{wc2010} showed that an ASP$(3,2r+1)$ with the basis $[-r,r]$ exists for any integer $r$ with $2\leq r\leq 100$ except for $r=4,5$. Mathon demonstrated \cite[Theorem 10]{mathon} that the existence of an $(n,4,1)$-PDF implies the existence of an ASP$(3,n)$ with the basis $[-(n-1)/2,(n-2)/2]$. Together with Lemma \ref{thm:4-PDF}, these results imply the following fact.


\begin{Theorem}\label{thm:ASP}
There exists an ASP$(3,n)$ with the basis $[-(n-1)/2,(n-1)/2]$ for any $n\equiv 1\pmod{12}$.
\end{Theorem}

For the early history and more results on additive sequences of permutations, see \cite{Abrham84,gms}.

\subsection{Perfect difference matrices}

Difference matrices are one of the central topics in combinatorial designs \cite{bjh,c07}. Perfect difference matrices, which are a special type of difference matrix, were introduced in \cite{glm} for the construction of radar arrays.

\begin{Definition}
Let $m$ be a positive integer and $n$ be a positive odd integer. An $m\times n$ matrix $D=(d_{ij})$ with entries from $\{0,1,\ldots,n-1\}$ is called a \emph{perfect difference matrix}, denoted by a PDM$(m,n)$, if for all $1\leq i_1<i_2\leq m$, the list of differences
$$
\{d_{i_1,j}-d_{i_2,j}: 1\leq j\leq n\}=[-(n-1)/2,(n-1)/2].
$$
Note that here, no arithmetic modulo $n$ is performed when calculating differences. A PDM$(m,n)$ is said to be \emph{homogeneous} when each of its rows contains all elements of $\{0,1,\ldots,n-1\}$.
\end{Definition}

\begin{Proposition}\label{lem:PDM-ASP} \emph{(\cite[Theorem 3.10]{gms}, \cite[Lemma 2.1]{cwf})}
For any positive odd integer $n$, a homogeneous PDM$(m,n)$ is equivalent to an ASP$(m,n)$ with the basis $[-(n-1)/2,(n-1)/2]$.
\end{Proposition}

By Proposition \ref{lem:PDM-ASP} and Theorem \ref{thm:ASP}, we have the following result.

\begin{Theorem}\label{thm:main-PDM}
A homogeneous PDM$(3,n)$ exists for any $n\equiv 1 \pmod{12}$.
\end{Theorem}

\subsection{Difference triangle sets}\label{sec:DTS}

Difference triangle sets are well-studied combinatorial structures \cite{s07} with significant applications in error-correcting codes, particularly in constructing convolutional self-orthogonal codes \cite{rb67} and designing sequences with good autocorrelation properties \cite{cc03}.

\begin{Definition}
Let $m$ and $k$ be positive integers. An $(m,k)$-\emph{difference triangle set} $($briefly DTS$)$ is a set $\mathcal{A}=\{A_1,\ldots,A_m\}$, where for each $i \in [m]$, $A_i=\{a_{i0},a_{i1},\ldots,a_{ik}\}$ with $a_{ij}$ an integer, such that all differences $a_{i\ell}-a_{ij}$ for $1 \leq i \leq m$ and $0 \leq j, \ell \leq k$ with $j \neq \ell$ are all distinct and nonzero. It is \emph{normalized} if for every $i \in [m]$, we have $0 = a_{i0} < a_{i1} < \cdots < a_{ik}$. The \emph{scope} of a normalized $(m,k)$-DTS $\mathcal{A}$ is defined as $s(\mathcal{A})=\max\{ a_{ik}: i\in [m]\}$.
\end{Definition}

Clearly for a normalized $(m,k)$-DTS $\mathcal{A}$, we have $s(\mathcal{A})\geq m\binom{k+1}{2}$. From a practical point of view, difference triangle sets with a smaller scope yield better results. Thus, the case where $s(\mathcal{A})$ equals $m\binom{k+1}{2}$ attracts significant interest \cite{sksk}. For example, a $(1,k)$-DTS with minimum scope is equivalent to a \emph{Golomb ruler} with $k+1$ marks, which have various applications ranging from radio astronomy to cryptography (cf. \cite{ad11}). Perfect difference families can produce difference triangle sets with the minimum possible scope. Currently, known existence results for minimum-scope $(m,k)$-DTSs are scarce when $k\geq 3$. An $(mk(k+1)+1,k+1,1)$-PDF naturally yields an $(m,k)$-DTS with scope $m\binom{k+1}{2}$. Therefore, by Lemma \ref{thm:4-PDF}, we have the following result.

\begin{Theorem}\label{thm:main-DTS}
There exists an $(m,3)$-DTS with minimum scope $6m$ for any $m\geq 1$ and $m\neq 2,3$.
\end{Theorem}

\subsection{Optical orthogonal codes}

An optical orthogonal code is a family of sequences with good auto- and cross-correlation properties, used for synchronization and user identification, respectively, in an optical code-division multiple access (OCDMA) system \cite{csw1989}.

\begin{Definition}
Let $n$ and $k$ be positive integers. An $(n,k,1)$-\emph{optical orthogonal code} $($briefly OOC$)$ is a family $\cal{C}$ of $(0,1)$ sequences $($called {\em codewords}$)$
of length $n$ and Hamming weight $k$ satisfying the following two properties:
\begin{enumerate}
\item[$(1)$] $($the auto-correlation$)$: for any
${\mathbf{x}}=\{x_i\}_{i=0}^{n-1}\in\cal{C}$ and any integer $r$, $r
\not\equiv 0\ ({\rm mod}\ n)$,
$\sum_{i=0}^{n-1}x_ix_{i+r}\leq 1;$
\item[$(2)$] $($the cross-correlation$)$: for any
${\mathbf{x}}=\{x_i\}_{i=0}^{n-1}\in\cal{C}$,
${\mathbf{y}}=\{y_i\}_{i=0}^{n-1}\in\cal{C}$ with
${\mathbf{x}}\not={\mathbf{y}}$ and any integer $r$,
$\sum_{i=0}^{n-1}x_iy_{i+r}\leq 1,$
\end{enumerate}
where all subscripts are reduced modulo $n$.
\end{Definition}

The number of codewords of an optical orthogonal code is called its \emph{size}. By analogy with the Johnson bound for constant-weight codes \cite{John62}, the largest possible size of an $(n,k,1)$-OOC is upper-bounded by $\lfloor (n-1)/(k(k-1))\rfloor$. An $(n,k,1)$-OOC is said to be \emph{J-optimal} if its size achieves this bound.


\begin{Proposition}\label{OOC-equiv} {\rm \cite[Theorem 2.1]{Yin98}}
A J-optimal $(n,k,1)$-OOC is equivalent to an $(n,k,1)$-CDP with $\lfloor (n-1)/(k(k-1))\rfloor$ base blocks.
\end{Proposition}

It was shown in \cite{bw1987,csw1989} that a J-optimal $(n,3,1)$-OOC exists if and only if $n\not\equiv 14,20 \pmod{24}$. Recently, Zhao, Chang and Feng \cite{zcf2025} established the existence of J-optimal $(n,4,1)$-OOCs by constructing $(n,4,1)$-CDPs with $\lfloor (n-1)/12\rfloor$ base blocks for all $n\neq 25$.


\begin{Theorem}\label{thm:OOC-4}{\rm \cite{zcf2025}}
A J-optimal $(n,4,1)$-OOC exists if and only if $n\neq 25$.
\end{Theorem}

Abel and Buratti \cite[Remark 1.4]{ab2004} observed that a $(v,4,1)$-PDF implies a J-optimal $(n,4,1)$-OOC for all $v \leq n < v+12$. Therefore, by applying Lemma \ref{thm:4-PDF}, we obtain a much shorter and more elegant proof of Theorem \ref{thm:OOC-4} for the case $n\geq 49$.

\subsection{Geometric orthogonal codes}

The study of geometric orthogonal codes was motivated by their application in DNA origami technology (cf. \cite{gwnd,roth}). To design sets of macrobonds in 3D DNA origami so as to reduce undesirable bonding arising from misalignment and mismatches, Doty and Winslow \cite{dw} introduced the concept of geometric orthogonal codes.

\begin{Definition}
Let $\mathbb{Z}$ be the set of integers. Let $n_1$, $n_2$ and $w$ be positive integers. An {\em $(n_1\times n_2,w,1)$-geometric orthogonal code} $($briefly GOC$)$ is a family $\cal{F}$ of $w$-subsets of $[0,n_1-1]\times[0,n_2-1]$ $($called \emph{codewords} or {\em macrobonds}$)$ such that:
\begin{itemize}
\item[$(1)$] $($the aperiodic auto-correlation$)$:
		$|B\cap(B+(s,t))|\leq 1$ for all $B\in\cal{F}$ and every $(s,t)\in \mathbb{Z}\times \mathbb{Z}\setminus\{(0,0)\}$;
\item[$(2)$] $($the aperiodic cross-correlation$)$:
		$|A\cap(B+(s,t))|\leq 1$ for all $A,B\in\cal{F}$ with $A\neq B$ and every $(s,t)\in\mathbb{Z}\times \mathbb{Z}$,
\end{itemize}
where $B+(s,t)=\{(x+s,y+t):(x,y)\in B\}$.
\end{Definition}

The number of codewords of an $(n_1\times n_2,w,1)$-GOC is called its \emph{size}. For fixed $n_1$, $n_2$, and $w$, denote by $\Phi(n_1\times n_2,w,1)$ the largest possible size among all $(n_1\times n_2,w,1)$-GOCs. An $(n_1\times n_2,w,1)$-GOC with $\Phi(n_1\times n_2,w,1)$ codewords is said to be \emph{optimal}.

An upper bound on the size of an optimal $(n_1\times n_2,w,1)$-GOC is provided in \cite{wcftw}.

\begin{Proposition}\label{lem:GOC-nece}\emph{\cite[Corollary 3.4]{wcftw}}
For $2\leq w\leq5$,
$$\Phi(n_1\times n_2,w,1)\leq\frac{4n_1n_2-2n_1-2n_2}{w(w-1)}.$$
For $w\geq 6$,
$$\Phi(n_1\times n_2,w,1)<\frac{4n_1n_2-2n_1-2n_2}{w(w-1)}.$$
\end{Proposition}

To construct geometric orthogonal codes, a new type of difference packing was introduced by Wang et al. \cite{wcftw}.

\begin{Definition}
Let $U_1$ and $U_2$ be sets of integers containing $0$ such that if $a\in U_j$ then $-a\in U_j$ for every $j=1,2$. For any $w$-subset $B \subseteq U_1 \times U_2$, define its \emph{difference list} as the multiset
$$\delta B:=\{(x_1-y_1,x_2-y_2):(x_1,x_2),(y_1,y_2)\in B,(x_1,x_2)\neq(y_1,y_2)\}.$$
An \emph{$(U_1\times U_2,w,1)$-geometrical difference packing} (briefly GDP) is a collection
$\mathcal{B}$ of $w$-subsets of $U_1\times U_2$ (called \emph{base blocks}) such that $\delta\mathcal{B}:=\bigcup_{B\in\mathcal{B}}\delta B$ covers each element of $U_1\times U_2\setminus \{(0,0)\}$ at most once.
\end{Definition}

Let $u_1$ and $u_2$ be positive odd integers. We use the notation $(u_1 \times u_2, w, 1)$-GDP to refer to a $([-(u_1-1)/2,(u_1-1)/2] \times [-(u_2-1)/2,(u_2-1)/2], w, 1)$-GDP for brevity.



\begin{Proposition}\label{prop:equi}\emph{\cite[Proposition 2.8]{wcftw}}
A $((2n_1-1)\times(2n_2-1),w,1)$-GDP with $b$ base blocks is equivalent to an $(n_1\times n_2,w,1)$-GOC with $b$ codewords. 
\end{Proposition}

The following construction is a special case of Constructions 4.1 and 4.4 in \cite{wcftw}. For completeness we give a sketch of its proof.

\begin{Construction}\label{cons:GDP}\emph{\cite{wcftw}}
If there exist a $(u_1,w,1)$-PDF, a $(u_2,w,1)$-PDF and a PDM$(w,u_2)$, then there exists a $(u_1\times u_2,w,1)$-GDP with $(u_1u_2-1)/(w(w-1))$ base blocks.
\end{Construction}

\begin{proof}
Let $\mathcal{F}_1$ be a $(u_1,w,1)$-PDF and $\mathcal{F}_2$ be a $(u_2,w,1)$-PDF. Without loss of generality, assume that each base block in $\mathcal F_1$ and $\mathcal F_2$ contains the element $0$. Let $D=(d_{ij})$ be a PDM$(w,u_2)$. For each $F_1=\{a_{i}:1\leq i\leq w\}\in \mathcal{F}_1$ and for the $j$-th column of $D$, construct a set
$$B_{F_1,j}=\{(a_i,d_{ij}-\frac{u_2-1}{2}):1\leq i\leq w\}.$$
For each $F_2=\{b_{i}:1\leq i\leq w\}\in \mathcal{F}_2$, construct a set
$$C_{F_2}=\{(0,b_i):1\leq i\leq w\}.$$
Let $\mathcal{B}_1=\{B_{F_1,j}:F_1\in \mathcal{F}_1, 1\leq j\leq u_2\}$ and $\mathcal{B}_2=\{C_{F_2}:F_2\in \mathcal{F}_2\}$.
Then $\mathcal{B}_1\cup\mathcal B_2$ forms a $(u_1\times u_2,w,1)$-GDP with $(u_1u_2-1)/(w(w-1))$ base blocks.
\end{proof}

For $w=3$, the exact value of $\Phi(n_1\times n_2,3,1)$ is determined in \cite[Theorem 1.3]{wcftw} for any positive integers $n_1$ and $n_2$. There is no systematic treatment on the case of $w=4$ in the literature.

\begin{Theorem}\label{thm:GOC-main}
Let $n_1, n_2 \equiv 1 \pmod{6}$ and $n_1, n_2 \not\in \{13, 19\}$. Then
$$\Phi(n_1\times n_2,4,1)=\frac{2n_1n_2-n_1-n_2}{6}.$$
\end{Theorem}

\begin{proof}
For $j=1,2$, Lemma \ref{thm:4-PDF} guarantees the existence of a $(2n_j-1,4,1)$-PDF for any $n_j\equiv 1\pmod{6}$ with $n_j\neq 13,19$. By Theorem \ref{thm:main-PDM}, a homogeneous PDM$(3,2n_2-1)$ exists for any $n_2\equiv 1 \pmod{6}$. Adding a row of all zeros to this PDM yields a PDM$(4,2n_2-1)$. Thus one can apply Construction \ref{cons:GDP} to obtain a $((2n_1-1)\times (2n_2-1),4,1)$-GDP with $(2n_1n_2-n_1-n_2)/6$ base blocks. By Proposition \ref{prop:equi}, this GDP is equivalent to an $(n_1\times n_2,4,1)$-GOC with $(2n_1n_2-n_1-n_2)/6$ codewords. Proposition \ref{lem:GOC-nece} confirms that this GOC is optimal.
\end{proof}

\subsection{Graceful graphs}

\begin{Definition}
A graph $G=(V(G),E(G))$ with $q$ edges is said to admit a \emph{graceful labeling} if there exists an injective function $f: V(G) \rightarrow \{0, 1, 2, \ldots, q\}$ such that the induced function $f': E(G) \rightarrow \{1, 2, \ldots, q\}$, defined by $f'(uv) = |f(u) - f(v)|$ for each edge $uv \in E(G)$, is bijective. A graph that admits a graceful labeling is called a \emph{graceful graph}.
\end{Definition}

Let $m$ and $k$ be positive integers. Following the notation of the survey \cite{Gallian} on graph labelings, we denote by $K_k^{(m)}$ the \emph{windmill graph}, which comprises $m$ copies of the complete graph $K_k$ sharing a single common vertex. The case  $k=4$ is referred to as a \emph{French $m$-windmill graph} \cite{b78}.

Bermond \cite{b78} observed that the gracefulness of $K_4^{(m)}$ is equivalent to the existence of a $(12m+1,4,1)$-PDF. This equivalence motivated Conjecture \ref{conj:Berm} in Section \ref{sec:intro}. As a direct consequence of Lemma \ref{thm:4-PDF}, we obtain the following result.
	
\begin{Theorem}\label{main-graceful}
The French $m$-windmill graph is graceful for all positive integers $m\neq 2,3$.
\end{Theorem}

\section{Concluding remarks}\label{con:conluding}

This paper introduced layered difference families to establish the existence spectrum for $(m,4,3)$-PSDSs and $(v,4,\lambda)$-PDFs. The significance of this framework lies in its ability to construct infinite classes of $(v,k,\lambda)$-CDPs with a large number of base blocks. This is achieved by requiring that every base block adheres to the specific form given in \eqref{eqn:general form} in Section \ref{sec:layered-DF}. Under this approach, the construction of an infinite class reduces to that of a small-order LDF, which is then subjected to further expansion. This methodology is powerful, as it condenses a problem of infinite scope into a finite and manageable one. Therefore, the following research directions are motivated: to further investigate the properties of LDFs and to construct more LDFs with a large base block size.

Proposition \ref{thm:PSDS-nece}(2) gives a necessary condition $m \geq 2c-1$ for the existence of an $(m,4,c)$-PSDS. This paper only provides constructions for the cases $c=1$ and $c=3$. Investigating the general existence problem for arbitrary $c$ remains an open problem. A promising approach would be to develop recursive constructions, such as those given in \cite{rogersaddition,w86} and the methods employed in \cite{rogers86} to establish the existence of $(m,3,c)$-PSDSs for arbitrary $c$.

Erd\H{o}s conjectured that, for every positive integer $\ell$, except for a finite number of them, there is a non-trivial perfect system of difference sets whose positive differences constitute exactly the set $\{1, 2, \dots, \ell\}$. This conjecture was proven by Laufer and Turgeon \cite{conjErdos} through explicit direct constructions, and later generalized by Rogers \cite{generalerdos}. In their work, the base blocks are predominantly of size $3$, with at most one block of size $4$ and two blocks of size $5$. An interesting question is whether an analogous result holds for perfect systems of difference sets composed entirely of base blocks whose size is strictly larger than 3.

Very recently, difference triangle sets were employed to construct staircase codes \cite{sksk}. Motivated by the relationship between difference triangle sets and perfect difference families shown in Section \ref{sec:DTS}, an intriguing question arises: for a $(12m+1,4,1)$-PDF, $\mathcal{A} = \{A_1, \ldots, A_m\}$, where each base block $A_j = \{0, a_{j1}, a_{j2}, a_{j3}\}$ for $j \in [m]$ satisfies $0 < a_{j1} < a_{j2} < a_{j3}$, can we construct such a PDF that achieves the minimum possible value for the sum $\sum_{j \in [m]} a_{j3}$? Specifically, the question is whether this minimum sum equals $5m^2 + m$ (see \cite[Proposition 6]{sksk}).

Finally, we remark that there is a similar concept to perfect systems of difference sets in the literature, named \textit{perfect systems of sets of iterated differences} \cite{hrr}. Earlier and related work is reported in \cite{Golay,Harborth,Leech}.


\appendix
\setcounter{equation}{0}
\renewcommand\theequation{A.\arabic{equation}}

\section{$(m,4,3)$-PSDSs for $5\leq m\leq 16$ in Lemma \ref{lem:small-PSDS}}\label{Appendix:small-PSDS}

\begin{longtable}{llllllll}
$m=5$: & $\{0, 3, 20, 28\}$,&$\{0, 4, 19, 31\}$,&$\{0, 5, 18, 29\}$,&$\{0, 6, 22, 32\}$,&$\{0, 7, 21, 30\}$. \\
$m=6$: & $\{0, 3, 28, 33\}$,&$\{0, 4, 24, 36\}$,&$\{0, 6, 22, 35\}$,&$\{0, 7, 21, 38\}$,&$\{0, 8, 23, 34\}$, \\
 & $\{0, 10, 19, 37\}$. \\
$m=7$: & $\{0, 3, 23, 37\}$,&$\{0, 4, 30, 40\}$,&$\{0, 5, 21, 43\}$,&$\{0, 6, 31, 39\}$,&$\{0, 7, 18, 42\}$, \\
 & $\{0, 9, 28, 41\}$, & $\{0, 15, 27, 44\}$.\\
$m=8$: & $\{0, 3, 26, 46\}$,&$\{0, 5, 30, 41\}$,&$\{0, 6, 28, 45\}$,&$\{0, 7, 38, 47\}$,&$\{0, 8, 32, 50\}$,\\
 & $\{0, 10, 29, 44\}$, & $\{0, 12, 16, 49\}$,&$\{0, 13, 27, 48\}$.	\\
$m=9$: & $\{0, 3, 37, 52\}$,&$\{0, 4, 42, 48\}$,&$\{0, 5, 18, 51\}$,&$\{0, 7, 27, 50\}$,&$\{0, 8, 30, 47\}$, \\
 & $\{0, 9, 40, 54\}$, & $\{0, 12, 28, 53\}$, &$\{0, 19, 29, 55\}$, &$\{0, 21, 32, 56\}$. \\
$m=10$: & $\{0, 3, 36, 57\}$,&$\{0, 4, 32, 45\}$,&$\{0, 5, 53, 60\}$,&$\{0, 6, 29, 49\}$,&$\{0, 8, 42, 52\}$,\\
 & $\{0, 9, 24, 59\}$, & $\{0, 11, 37, 62\}$, &$\{0, 12, 39, 58\}$,	&$\{0, 14, 31, 61\}$, &$\{0, 16, 38, 56\}$.\\
$m=11$: & $\{0, 3, 38, 59\}$, &$\{0, 4, 43, 68\}$, &$\{0, 5, 29, 62\}$, &$\{0, 6, 42, 53\}$, &$\{0, 8, 48, 66\}$, \\
 & $\{0, 9, 41, 63\}$, & $\{0, 10, 27, 55\}$, &$\{0, 12, 31, 61\}$, &$\{0, 13, 50, 65\}$, &$\{0, 14, 34, 60\}$, \\
 & $\{0, 16, 23, 67\}$. \\
$m=12$: & $\{0, 3, 48, 63\}$, &$\{0, 4, 37, 59\}$, &$\{0, 5, 23, 72\}$, &$\{0, 6, 35, 56\}$, &$\{0, 7, 53, 69\}$, \\
 & $\{0, 8, 39, 73\}$, & $\{0, 9, 51, 70\}$, & $\{0, 10, 36, 74\}$, &$\{0, 11, 43, 68\}$, &$\{0, 12, 52, 66\}$, \\
 & $\{0, 17, 30, 58\}$, &$\{0, 24, 44, 71\}$. \\
$m=13$: & $\{0, 3, 68, 72\}$, &$\{0, 6, 55, 79\}$, &$\{0, 7, 50, 63\}$, &$\{0, 8, 45, 67\}$, &$\{0, 10, 44, 76\}$, \\
 & $\{0, 11, 29, 62\}$, & $\{0, 12, 42, 70\}$, &$\{0, 14, 60, 75\}$, &$\{0, 16, 39, 80\}$, &$\{0, 17, 36, 71\}$, \\
 & $\{0, 20, 25, 77\}$, &$\{0, 21, 48, 74\}$, &	$\{0, 31, 40, 78\}$. \\
$m=14$: & $\{0, 3, 66, 83\}$, &$\{0, 4, 18, 60\}$, &$\{0, 5, 39, 84\}$, &$\{0, 6, 58, 74\}$, &$\{0, 7, 32, 82\}$, \\
 & $\{0, 8, 59, 69\}$, & $\{0, 9, 55, 81\}$,&$\{0, 11, 47, 78\}$,&$\{0, 12, 49, 77\}$,&$\{0, 15, 48, 86\}$,\\
 & $\{0, 19, 54, 76\}$,&$\{0, 20, 44, 73\}$,& $\{0, 21, 62, 85\}$,&$\{0, 27, 40, 70\}$. \\
$m=15$: & $\{0, 3, 31, 80\}$,&$\{0, 4, 44, 61\}$,&$\{0, 5, 69, 78\}$,&$\{0, 6, 58, 74\}$,&$\{0, 7, 36, 92\}$,\\
 & $\{0, 8, 70, 83\}$,&	$\{0, 10, 43, 91\}$,&$\{0, 11, 76, 90\}$,&$\{0, 12, 51, 72\}$,&$\{0, 15, 45, 86\}$, \\
 & $\{0, 18, 38, 84\}$,&$\{0, 22, 47, 89\}$,& $\{0, 23, 55, 82\}$,&$\{0, 24, 50, 87\}$,&$\{0, 34, 53, 88\}$.	\\
$m=16$: & $\{0, 3, 40, 72\}$,&$\{0, 4, 23, 91\}$,&$\{0, 5, 49, 65\}$,&$\{0, 6, 48, 79\}$,&$\{0, 7, 46, 96\}$,\\
 & $\{0, 8, 61, 82\}$,&	$\{0, 9, 84, 95\}$,&$\{0, 10, 62, 80\}$,&$\{0, 12, 57, 90\}$,&$\{0, 13, 56, 94\}$, \\
 & $\{0, 15, 51, 98\}$,&$\{0, 17, 71, 93\}$, & $\{0, 20, 34, 97\}$, &$\{0, 24, 59, 88\}$, &$\{0, 25, 66, 92\}$,\\
 & $\{0, 27, 55, 85\}$.	
\end{longtable}

\section{The remaining $18+x$ base blocks in Lemma \ref{lem:other-PSDS}}\label{Appendix:main}

{\setlength{\tabcolsep}{0.3pt}
\begin{longtable}{|lll|}
	\hline
	$x = 0$ & & \\
	\hline
	$\{0, 3t-3, 38t, 48t-1\}$
	&$\{0, 14t-1, 30t+1, 53t-2\}$
	&$\{0, 6t, 44t-2, 50t-1\}$
	\\
	$\{0, 11t, 34t-2, 45t-1\}$
	&$\{0, 3t-1, 32t-1, 50t-2\}$
	&$\{0, 12t+2, 20t+2, 54t+2\}$
	\\
	$\{0, 4t-1, 29t+1, 39t+1\}$
	&$\{0, 11t-1, 27t, 41t+1\}$
	&$\{0, 3t-2, 38t-1, 54t-1\}$
	\\
	$\{0, 12t, 40t+1, 54t+1\}$
	&$\{0, 8t-1, 28t-1, 40t\}$
	&$\{0, 3t, 48t+1, 53t\}$
	\\
	$\{0, 9t, 34t+1, 54t\}$
	&$\{0, 16t-1, 24t, 51t-1\}$
	&$\{0, 23t-1, 28t, 53t-1\}$
	\\
	$\{0, 5t, 24t-1, 44t\}$
	&$\{0, 4t, 19t, 51t\}$
	&$\{0, 7t, 25t, 48t\}$\\
	\hline
	
	$x = 1$ & & \\
	\hline
	$\{0, 3t-3, 38t, 44t+2\}$
	&$\{0, 15t+3, 38t+4, 54t+8\}$
	&$\{0, 11t+2, 34t+1, 45t+4\}$
	\\
	$\{0, 19t+2, 25t+5, 54t+7\}$
	&$\{0, 6t, 38t+1, 54t+4\}$
	&$\{0, 3t-2, 14t-1, 53t+3\}$
	\\
	$\{0, 3t-1, 23t+3, 48t+5\}$
	&$\{0, 18t+1, 27t+3, 50t+3\}$
	&$\{0, 4t+1, 32t+3, 51t+6\}$
	\\
	$\{0, 4t, 27t+2, 51t+4\}$
	&$\{0, 9t, 41t+4, 51t+5\}$
	&$\{0, 14t+2, 30t+4, 48t+6\}$
	\\
	$\{0, 8t, 38t+3, 53t+5\}$
	&$\{0, 12t+2, 20t+3, 54t+6\}$
	&$\{0, 4t-1, 29t+3, 44t+4\}$
	\\
	$\{0, 14t+1, 19t+1, 54t+5\}$
	&$\{0, 11t, 14t, 39t+3\}$
	&$\{0, 5t+1, 30t+2, 50t+4\}$
	\\
	$\{0, 6t+1, 44t+3, 53t+4\}$
	&
	&
	\\
	\hline
	
	$x = 2$ & & \\
	\hline
	$\{0, 15t+5, 29t+5, 53t+9\}$
	&$\{0, 16t+6, 35t+9, 54t+14\}$
	&$\{0, 3t-3, 23t+2, 38t+5\}$
	\\
	$\{0, 9t+3, 18t+3, 32t+4\}$
	&$\{0, 20t+6, 23t+4, 48t+12\}$
	&$\{0, 15t+4, 29t+6, 54t+13\}$
	\\
	$\{0, 7t+2, 32t+8, 51t+12\}$
	&$\{0, 11t+3, 34t+6, 53t+12\}$
	&$\{0, 16t+5, 27t+7, 54t+12\}$
	\\
	$\{0, 4t+1, 34t+7, 45t+11\}$
	&$\{0, 9t+1, 39t+8, 51t+11\}$
	&$\{0, 4t+2, 14t+4, 48t+9\}$
	\\
	$\{0, 18t+4, 25t+5, 50t+9\}$
	&$\{0, 5t+1, 45t+10, 53t+11\}$
	&$\{0, 3t-1, 35t+6, 51t+10\}$
	\\
	$\{0, 3t+1, 27t+6, 50t+11\}$
	&$\{0, 6t, 38t+6, 47t+8\}$
	&$\{0, 3t, 47t+9, 53t+10\}$
	\\
	$\{0, 4t, 39t+7, 45t+9\}$
	&$\{0, 10t+3, 40t+8, 54t+11\}$
	&
	\\
	\hline
	
	$x = 3$ & & \\
	\hline
	$\{0, 11t+5, 34t+8, 45t+14\}$
	&$\{0, 6t+5, 29t+9, 44t+13\}$
	&$\{0, 16t+8, 19t+5, 54t+18\}$
	\\
	$\{0, 20t+8, 23t+6, 48t+17\}$
	&$\{0, 6t+4, 12t+6, 54t+20\}$
	&$\{0, 6t+1, 38t+9, 53t+15\}$
	\\
	$\{0, 3t-1, 47t+13, 50t+14\}$
	&$\{0, 3t, 28t+10, 51t+15\}$
	&$\{0, 10t+2, 14t+3, 44t+12\}$
	\\
	$\{0, 14t+5, 38t+11, 53t+16\}$
	&$\{0, 6t+3, 38t+12, 54t+17\}$
	&$\{0, 19t+6, 27t+9, 51t+16\}$
	\\
	$\{0, 9t+2, 39t+12, 50t+15\}$
	&$\{0, 12t+5, 39t+13, 53t+17\}$
	&$\{0, 4t+3, 29t+10, 54t+19\}$
	\\
	$\{0, 4t+2, 24t+8, 44t+15\}$
	&$\{0, 5t+2, 23t+7, 39t+14\}$
	&$\{0, 4t, 34t+11, 45t+15\}$
	\\
	$\{0, 10t+3, 35t+11, 51t+17\}$
	&$\{0, 9t+3, 19t+7, 51t+18\}$
	&$\{0, 8t+2, 35t+12, 53t+18\}$
	\\
	\hline
	
	$x = 4$ & & \\
	\hline
	$\{0, 4t, 15t+7, 38t+12\}$
	&$\{0, 3t-3, 14t+3, 38t+13\}$
	&$\{0, 15t+8, 38t+14, 54t+24\}$
	\\
	$\{0, 3t-1, 44t+16, 48t+19\}$
	&$\{0, 19t+7, 23t+8, 53t+20\}$
	&$\{0, 12t+8, 35t+17, 54t+26\}$
	\\
	$\{0, 3t-2, 32t+11, 51t+21\}$
	&$\{0, 6t+4, 29t+11, 54t+25\}$
	&$\{0, 3t, 47t+18, 53t+21\}$
	\\
	$\{0, 7t+3, 27t+13, 54t+22\}$
	&$\{0, 6t+2, 44t+17, 53t+22\}$
	&$\{0, 5t+2, 32t+13, 50t+20\}$
	\\
	$\{0, 10t+5, 35t+14, 51t+23\}$
	&$\{0, 8t+3, 16t+8, 48t+20\}$
	&$\{0, 4t+2, 38t+17, 45t+21\}$
	\\
	$\{0, 8t+4, 23t+10, 48t+22\}$
	&$\{0, 16t+7, 19t+8, 51t+22\}$
	&$\{0, 11t+4, 38t+16, 50t+22\}$
	\\
	$\{0, 10t+4, 28t+12, 42t+19\}$
	&$\{0, 9t+4, 39t+17, 53t+23\}$
	&$\{0, 24t+9, 29t+12, 54t+23\}$
	\\
	$\{0, 11t+5, 25t+10, 45t+19\}$
	&
	&
	\\
	\hline
	
	$x = 5$ & & \\
	\hline			
	$\{0, 3t-3, 14t+4, 44t+21\}$
	&$\{0, 15t+8, 38t+16, 53t+26\}$
	&$\{0, 12t+9, 35t+21, 54t+32\}$
	\\
	$\{0, 6t+5, 29t+14, 45t+25\}$
	&$\{0, 3t-2, 23t+11, 35t+17\}$
	&$\{0, 6t+3, 38t+17, 54t+29\}$
	\\
	$\{0, 15t+9, 20t+12, 47t+24\}$
	&$\{0, 3t-1, 47t+23, 51t+26\}$
	&$\{0, 9t+4, 34t+16, 50t+25\}$
	\\
	$\{0, 4t+4, 27t+14, 54t+30\}$
	&$\{0, 3t+2, 35t+19, 54t+31\}$
	&$\{0, 11t+5, 18t+9, 53t+29\}$
	\\
	$\{0, 5t+4, 16t+10, 45t+26\}$
	&$\{0, 3t, 28t+15, 51t+28\}$
	&$\{0, 4t+1, 38t+19, 45t+24\}$
	\\
	$\{0, 8t+4, 20t+11, 47t+26\}$
	&$\{0, 8t+5, 35t+18, 53t+28\}$
	&$\{0, 4t+2, 19t+9, 51t+27\}$
	\\
	$\{0, 3t+1, 32t+16, 41t+22\}$
	&$\{0, 6t+4, 25t+14, 50t+27\}$
	&$\{0, 9t+5, 39t+21, 53t+27\}$
	\\
	$\{0, 14t+7, 24t+13, 54t+28\}$
	&$\{0, 10t+5, 34t+17, 48t+25\}$
	&
	\\
	\hline
	
	$x = 6$ & & \\
	\hline
	$\{0, 10t+8, 25t+20, 48t+30\}$
	&$\{0, 16t+14, 19t+11, 54t+35\}$
	&$\{0, 3t-1, 38t+20, 44t+26\}$
	\\
	$\{0, 19t+13, 30t+21, 53t+32\}$
	&$\{0, 3t-2, 14t+7, 38t+23\}$
	&$\{0, 9t+6, 23t+12, 50t+32\}$
	\\
	$\{0, 15t+10, 24t+15, 47t+28\}$
	&$\{0, 8t+6, 20t+15, 47t+30\}$
	&$\{0, 15t+11, 38t+25, 54t+38\}$
	\\
	$\{0, 4t+1, 34t+20, 54t+34\}$
	&$\{0, 3t, 47t+29, 51t+32\}$
	&$\{0, 7t+5, 32t+22, 48t+34\}$
	\\
	$\{0, 3t+2, 45t+31, 51t+35\}$
	&$\{0, 5t+3, 32t+20, 50t+31\}$
	&$\{0, 14t+8, 34t+21, 53t+33\}$
	\\
	$\{0, 19t+14, 27t+19, 51t+33\}$
	&$\{0, 6t+3, 44t+27, 53t+34\}$
	&$\{0, 14t+11, 25t+18, 54t+37\}$
	\\
	$\{0, 4t+2, 18t+12, 45t+30\}$
	&$\{0, 3t+1, 10t+7, 35t+22\}$
	&$\{0, 12t+8, 40t+27, 54t+36\}$
	\\
	$\{0, 5t+4, 30t+20, 53t+35\}$
	&$\{0, 5t+2, 34t+22, 44t+28\}$
	&$\{0, 6t+5, 35t+23, 51t+34\}$
	\\
	\hline
	
	$x = 7$ & & \\
	\hline
	$\{0, 15t+13, 38t+25, 53t+39\}$
	&$\{0, 3t-3, 14t+8, 23t+14\}$
	&$\{0, 3t-1, 38t+24, 54t+40\}$
	\\
	$\{0, 16t+15, 25t+23, 48t+36\}$
	&$\{0, 3t-2, 38t+27, 44t+31\}$
	&$\{0, 3t, 44t+30, 50t+36\}$
	\\
	$\{0, 6t+5, 35t+28, 54t+44\}$
	&$\{0, 20t+15, 24t+16, 47t+33\}$
	&$\{0, 3t+1, 19t+15, 53t+38\}$
	\\
	$\{0, 7t+7, 39t+30, 54t+41\}$
	&$\{0, 4t+2, 23t+15, 51t+37\}$
	&$\{0, 19t+14, 24t+17, 51t+40\}$
	\\
	$\{0, 10t+7, 24t+18, 51t+38\}$
	&$\{0, 4t+3, 34t+25, 48t+35\}$
	&$\{0, 12t+11, 30t+24, 44t+33\}$
	\\
	$\{0, 5t+4, 39t+28, 50t+38\}$
	&$\{0, 8t+7, 28t+23, 53t+40\}$
	&$\{0, 10t+9, 25t+21, 48t+37\}$
	\\
	$\{0, 9t+7, 27t+21, 54t+43\}$
	&$\{0, 14t+12, 25t+20, 54t+42\}$
	&$\{0, 8t+6, 42t+32, 53t+41\}$
	\\
	$\{0, 7t+6, 32t+25, 48t+38\}$
	&$\{0, 3t+2, 38t+29, 42t+33\}$
	&$\{0, 5t+5, 30t+23, 40t+31\}$
	\\
	$\{0, 7t+5, 39t+29, 51t+39\}$
	&
	&
	\\
	\hline
\end{longtable}}

\section{$(v,4,2)$-PDFs for small values of $v$ in Lemma \ref{thm:(v,4,2)PDF}} \label{Appendix:small-PDF-2}

\begin{longtable}{llllllll}
$v=19$: & $\{0,1,4,9\}$,& $\{0,2,6,9\}$,& $\{0,2,7,8\}$. \\
$v=25$: & $\{0,1,4,10\}$,& $\{0,2,7,12\}$,& $\{0,2,8,11\}$,& $\{0,4,11,12\}$. \\
$v=31$: &$\{0,1,4,11\}$,& $\{0,2,8,14\}$,& $\{0,2,11,15\}$,& $\{0,3,10,15\}$,& $\{0,5,13,14\}$. \\
$v=37$: & $\{0,3,10,18\}$,& $\{0,3,13,17\}$,& $\{0,4,15,16\}$,& $\{0,2,7,16\}$,& $\{0,5,6,18\}$,\\
& $\{0,6,8,17\}$. \\
$v=43$: & $\{0,4,10,20\}$,& $\{0,5,12,21\}$,& $\{0,6,14,19\}$,& $\{0,2,11,13\}$,& $\{0,3,17,18\}$,\\
& $\{0,3,20,21\}$,& $\{0,4,12,19\}$. \\
$v=55$: & $\{0,4,22,26\}$,& $\{0,5,17,26\}$,& $\{0,6,14,25\}$,& $\{0,6,16,27\}$,& $\{0,9,23,24\}$,\\
&$\{0,2,5,18\}$,& $\{0,2,10,25\}$,& $\{0,3,20,27\}$,& $\{0,7,19,20\}$.
\end{longtable}

\section{The remaining $12+x$ base blocks in Lemma \ref{thm:(v,4,2)PDF}}\label{Appendix:big-PDF-2}

\begin{longtable}{|lll|}
		\hline
		$x = -1$ & & \\
		\hline
		$\{0, 18t-5, 3t-4, 11t-3\} $
		& $\{0, 18t-4, 4t-1, 4t\} $
		& $\{0, 18t-3, 8t-2, 17t-4\} $ \\
		$\{0, 9t, 17t-3, 12t-1\} $
		& $\{0, 3t, 15t, 10t+1\} $
		& $\{0, 4t-2, 15t-2, 14t-2\} $ \\
		$\{0, 14t, 15t-1, t+1\} $
		& $\{0, 8t, 11t-2, 14t-1\} $
		& $\{0, 17t-1, 11t-1, 6t-2\} $ \\
		$\{0, 9t-1, 17t-2, 12t-2\} $
		& $\{0, 4t+1, 7t-1, 14t-1\} $
		& \\
		\hline
		$x = 1$ & & \\
		\hline
		$\{0, 18t+2, 4t+4, 7t\} $
		& $\{0, 18t-1, 17t, 6t-1\} $
		& $\{0, 2, 15t+2, 14t+1\} $\\
		$\{0, 18t, 8t-1, 15t+1\} $
		& $\{0, 9t, 18t+1, 3t-2\} $
		& $\{0, 12t+1, 7t+3, t+2\} $\\
		$\{0, 17t+1, 3t+1, 10t\} $
		& $\{0, 4t+2, 12t+3, 15t+2\} $
		& $\{0, 18t+3, 13t+2, 5t-1\} $\\
		$\{0, 9t+2, 17t+4, 12t+2\} $
		& $\{0, 1, 4t, 17t+3\} $
		& $\{0, 18t+3, 8t, 3t\} $\\
		$\{0, 4t+1, 15t+4, 14t+3\} $
		& &\\
		\hline
		$x = 3$ & & \\
		\hline
		$\{0, 2, 13t+8, 4t+6\} $
		& $\{0, t+3, 18t+7, 15t+5\} $
		& $\{0, 1, 18t+5, 10t+4\} $\\
		$\{0, t, 18t+6, 5t-1\} $
		& $\{0, 6t+1, 7t+2, 14t+3\} $
		& $\{0, 17t+5, 12t+4, 3t-1\} $\\
		$\{0, t+2, 15t+6, 10t+6\} $
		& $\{0, 18t+9, 3t, 14t+7\} $
		& $\{0, 18t+8, 13t+5, 3t\} $\\
		$\{0, 1, 17t+9, 8t+5\} $
		& $\{0, 11t+5, 4t, 15t+6\} $
		& $\{0, 11t+5, t+1, 18t+8\} $\\
		$\{0, 6t+3, 12t+5, 18t+9\} $
		& $\{0, t+2, 15t+7, 10t+5\} $
		& $\{0, 11t+4, 8t+3, 15t+7\} $ \\\hline
	\end{longtable}

\section{$(v,4,3)$-PDFs for small values of $v$ in Lemma \ref{thm:(v,4,3)PDF}}
\label{Appendix:small-PDF-3}

\begin{longtable}{llllllll}
$v=17$: & $\{0,1,4,8\}$,& $\{0,2,5,8\}$,& $\{0,2,6,7\}$,& $\{0,2,7,8\}$. \\
$v=21$: & $\{0,1,4,9\}$,& $\{0,2,6,10\}$,& $\{0,2,8,9\}$,& $\{0,3,5,10\}$,& $\{0,3,9,10\}$. \\
$v=25$: & $\{0,1,3,10\}$,& $\{0,2,6,11\}$,& $\{0,2,8,12\}$,& $\{0,3,8,12\}$,& $\{0,3,10,11\}$, \\
& $\{0,5,11,12\}$. \\
$v=29$: & $\{0,3,7,14\}$,& $\{0,3,8,14\}$,& $\{0,4,9,14\}$,& $\{0,1,12,13\}$,& $\{0,2,8,10\}$,
\\&$\{0,3,12,13\}$,& $\{0,4,6,13\}$. \\
$v=33$: & $\{0,3,8,16\}$,& $\{0,4,9,16\}$,& $\{0,4,10,16\}$,& $\{0,1,2,15\}$,& $\{0,2,11,13\}$, \\
&$\{0,3,10,14\}$,& $\{0,3,10,15\}$,& $\{0,6,14,15\}$. \\
$v=37$: & $\{0,4,10,18\}$,& $\{0,4,11,18\}$,& $\{0,5,11,18\}$,& $\{0,1,2,17\}$, & $\{0,2,12,15\}$, \\
&$\{0,2,12,17\}$, & $\{0,3,9,14\}$,& $\{0,3,12,16\}$,& $\{0,8,16,17\}$. \\
\end{longtable}

\section{The remaining $12+x$ base blocks in Lemma \ref{thm:(v,4,3)PDF}}\label{Appendix:big-PDF-3}

\begin{longtable}{|lll|}\hline
		$x=-2$ & &\\
		\hline
		$\{0, 2, 6t+1, 10t-2\} $
		& $\{0, 4t-4, 12t-5, 7t-3\} $
		& $\{0, 4t-2, 2t-3, 10t-5\} $\\
		$\{0, 4t, 12t-4, 6t\} $
		& $\{0, 1, 7t-1, 10t-1\} $
		& $\{0, 8t, 11t-3, t\} $\\
		$\{0, 4t-1, 12t-4, 10t-3\} $
		& $\{0, 8t-2, 10t-3, 5t\} $
		& $\{0, 2, 11t-2, 6t-1\} $\\
		$\{0, 1, 9t-2, 3t\} $ && \\
		\hline
		$x=-1$ & &\\
		\hline
		$\{0, 12t-2, 2t-1, 11t-2\} $
		& $\{0, 9t-2, 3t-1, 12t-2\} $
		& $\{0, 8t-1, 2t-2, 6t-3\} $\\
		$\{0, 2, 8t-1, 10t\} $
		& $\{0, 8t-2, 11t-1, 5t+1\} $
		& $\{0, 9t, 3t-2, 10t-1\} $\\
		$\{0, 9t, 4t-2, 6t-2\} $
		& $\{0, 4t, 11t, 10t-1\} $
		& $\{0, 1, 8t+1, 5t\} $\\
		$\{0, 4t+1, 2t+1, 10t+1\} $
		& $\{0, 1, 6t+1, 10t+1\} $ &\\
		\hline
		$x=1$ & &\\
		\hline
		$\{0, 2, 3t+1, 6t-1\} $
		& $\{0, 8t+4, 3t, 10t+2\} $
		& $\{0, 8t+3, 11t, t-1\} $\\
		$\{0, 4t+2, 12t+2, 6t+4\} $
		& $\{0, 1, 12t+1, 6t-1\} $
		& $\{0, 8t+3, 3t, 10t+3\} $\\
		$\{0, 4t+1, 12t, 10t+1\} $
		& $\{0, 9t, 12t+2, t+1\} $
		& $\{0, 4t+1, 12t+2, 11t+2\} $\\
		$\{0, 1, 6t+2, 10t+2\} $
		& $\{0, 4t, 12t+1, 10t\} $
		& $\{0, 9t+1, 10t+3, 5t+2\} $\\
		$\{0, 4t-1, 12t+1, 10t\} $ &&\\
		\hline
		$x=2$ & &\\
		\hline
		$\{0, 9t+3, 11t+5, 6t+6\} $
		& $\{0, 12t+1, 6t-2, 6t\} $
		& $\{0, 8t+4, 11t+4, t-1\} $\\
		$\{0, 4t-1, 12t+4, 7t+1\} $
		& $\{0, 1, 3t-1, 10t+3\} $
		& $\{0, 8t+3, 3t-1, 11t+3\} $\\
		$\{0, 9t+2, 12t+3, 6t-1\} $
		& $\{0, 12t+3, 2t-1, 6t-1\} $
		& $\{0, 4t+1, 12t+2, 10t+4\} $\\
		$\{0, 12t+4, 2t, 10t+1\} $
		& $\{0, 4t+2, 12t+4, 10t+3\} $
		& $\{0, 1, 10t+3, 5t+1\} $\\
		$\{0, 12t+2, 2t, 6t\} $
		& $\{0, 8t+3, 7t+2, t\} $ &\\
		\hline
	\end{longtable}

\section{The remaining $12+x$ base blocks in Lemma \ref{thm:(v,4,6)PDF}}\label{Appendix:big-PDF-6}

\begin{longtable}{|lll|}
		\hline
		$x=-1$ & &\\
		\hline
		$\{0, 2, 6t-1, t+1\} $
		& $\{0, 1, 4t-2, 6t-1\} $
		& $\{0, 1, 4t, 6t-1\} $\\
		$\{0, 4t-1, 3t-3, 6t-1\} $
		& $\{0, 1, 4t-1, 4t+1\} $
		& $\{0, 4t, 2t-1, 5t\} $\\
		$\{0, 3t-1, 6t-1, 5t-1\} $
		& $\{0, 3t-2, 3t+1, 5t\} $
		& $\{0, 1, 3t+2, 5t\} $\\
		$\{0, 2, 3t, 5t+1\} $
		& $\{0, 1, 3t, 5t\} $ &\\
		\hline
		$x=1$ & &\\
		\hline
		$\{0, 3t-3, 3t+1, 5t-2\} $
		& $\{0, 3, 4t, 6t\} $
		& $\{0, 2, 6t, 5t+1\} $\\
		$\{0, 3, 3t+1, 5t+2\} $
		& $\{0, 4t-2, 6t, 5t\} $
		& $\{0, 4t+1, 2t-1, 6t+1\} $\\
		$\{0, 2, 3t+2, 6t+1\} $
		& $\{0, 2, 6t+1, 5t+1\} $
		& $\{0, 4t-1, 2t+1, 6t+1\} $\\
		$\{0, 1, 6t+1, 5t\} $
		& $\{0, 1, 3t+2, 5t+1\} $
		& $\{0, 1, 4t, 6t+1\} $\\
		$\{0, 1, 3t+1, 5t+1\} $ & &\\
		\hline
	\end{longtable}

\end{document}